\documentclass[11pt]{article}
\usepackage[margin=0.9in]{geometry}

\usepackage{lineno,hyperref}
\modulolinenumbers[5]

\usepackage{latexsym, epsfig, amssymb, amsmath, amsthm, graphicx, mathrsfs,color}
\usepackage{anysize,natbib}
\usepackage[OT1]{fontenc}
\usepackage{multirow}

\marginsize{1.2in}{1in}{0.55in}{1.55in} %
\makeatletter

\usepackage{moreverb}

\usepackage{natbib}

\usepackage{animate}
\usepackage{multimedia}
\usepackage{bbm}
\usepackage{amsmath,amsthm}
\usepackage{subfigure}
\usepackage{graphicx}

\newcommand\BibTeX{{\rmfamily B\kern-.05em \textsc{i\kern-.025em b}\kern-.08em
T\kern-.1667em\lower.7ex\hbox{E}\kern-.125emX}}

\DeclareMathOperator*{\argmin}{arg\,min}

\newtheorem{lemma}{Lemma}
\newtheorem{proposition}{Proposition}
\newtheorem{theorem}{Theorem}
\newtheorem{remark}{Remark}

\makeatother

\def\boxit#1{\vbox{\hrule\hbox{\vrule\kern6pt\vbox{\kern6pt#1\kern6pt}\kern6pt\vrule}\hrule}}

\begin{document}

\title{On the estimation of correlation in a binary sequence model}

\author{
Haolei Weng  \vspace{0.2cm} \\
Department of Statistics and Probability, \\
Michigan State University, East Lansing, MI 48824  \vspace{0.4cm} \\
Yang Feng  \vspace{0.2cm} \\
Department of Biostatistics\\
College of Global Public Health\\
New York University, New York, NY 10003
}

\date{}
\maketitle

\begin{abstract}
We consider a binary sequence generated by thresholding a hidden continuous sequence. The hidden variables are assumed to have a compound symmetry covariance structure with a single parameter characterizing the common correlation. We study the parameter estimation problem under such one-parameter models. We demonstrate that maximizing the likelihood function does not yield consistent estimates for the correlation. We then formally prove the nonestimability of the parameter by deriving a non-vanishing minimax lower bound. This counter-intuitive phenomenon provides an interesting insight that one-bit information of each latent variable is not sufficient to consistently recover their common correlation. On the other hand, we further show that trinary data generated from the hidden variables can consistently estimate the correlation with parametric convergence rate. Thus we reveal a phase transition phenomenon regarding the discretization of latent continuous variables  while preserving the estimability of the correlation. Numerical experiments are performed to validate the conclusions. 
\end{abstract}

\noindent {\bf Key Words: binary data; consistency; maximum likelihood estimate; minimax lower bound; phase transition; thresholding; trinary data}

\section{Introduction}

Data sets consisting of dependent binary outcomes are common in quantitative fields. For instance, in longitudinal studies, monthly presence or absence of a disease for residents in a neighborhood might be recorded over a certain period; in social network analyses, linkages between individuals or  organizations on a social media website can be accessible;  in financial econometrics, the rise or fall of stock prices from the same sectors are observed. The existence of dependence makes it subtle to analyze such types of data. Several sophisticated modeling frameworks for dependent binary data have been well developed over the last two decades, such as graphical models \citep{koller2009probabilistic} and exponential random graph models \citep{lusher2013exponential}. 

In this paper, however, we consider an alternative and more transparent modeling method. Specifically, let $A=(A_1,\ldots, A_n)$ be the binary observations. We assume the data is generated from thresholding a latent continuous vector $X=(X_1,\ldots. X_n)$ with $\mathbb{E}(X)=\mu$ and $\mbox{cov}(X)=\Sigma$:
\begin{align}\label{model:formula}
A_i=\mathbbm{1}_{X_i>\tau},~~i=1,\ldots,n,
\end{align}
where $\tau$ is a pre-defined threshold. The dependency of $A$  can then be modeled by specific structures  on $\Sigma$. The idea of thresholding continuous variables to obtain discrete ones has been widely adopted in binary and ordinal regression problems \citep{mccullagh1980regression, mccullagh1989generalized}. For example, in probit models, each binary response variable is obtained by truncating an independent latent normal variable. Nevertheless, different from regular regression problems,  the model formulation \eqref{model:formula} puts emphasis on the dependency structure of the observations.

We use network data example to elaborate on the model formulation \eqref{model:formula}. For notational convenience, we rewrite the observations $\{A_i\}_{i=1}^n$ by a matrix $\mathcal{A}=(\mathcal{A}_{ij})_{m \times m}$ and the latent continuous variables $\{X_i\}_{i=1}^n$ by $\mathcal{X}=(\mathcal{X}_{ij})_{m \times m}$. The binary entry $\mathcal{A}_{ij}$ represents whether there exists an edge between nodes $i$ and $j$. We consider undirected networks where $\mathcal{A}_{ij}=\mathcal{A}_{ji}$ for simplicity. The covariance matrix $\Sigma$ belongs to $\mathbb{R}^{m^2 \times m^2}$. Due to rich structures exhibited in different types of networked systems \citep{newman2003structure}, there has been an extensive literature on network modeling including Erd{\"o}s--R{\'e}nyi random graph model \citep{erdos1959renyi}, exponential random graph model \citep{robins2007introduction}, stochastic blockmodel \citep{holland1983stochastic}, and latent space model \citep{hoff2002latent}, among others. As an alternative, model \eqref{model:formula} assumes the edges between nodes are generated from some underlying continuous variables, and the covariance matrix $\Sigma$ captures the possible dependency among different edges. One specification is
\begin{eqnarray*}
&& \mathcal{A}_{ij}=\mathbbm{1}_{\mathcal{X}_{ij}>0}, ~~~i=1,\ldots,m, ~j=1,\ldots, m, \\
&&\mathcal{X}=(\mathcal{X}_{ij})_{m\times m} \sim N(\Theta, \Sigma).
\end{eqnarray*}
The mean parameter $\Theta \in \mathbb{R}^{m\times m}$ incorporates heterogeneity across different edges. It can be assumed of low-rank based on the hypothesis that the generation mechanism of edges is driven by a few node-specific factors.  This is in the same spirit of both stochastic block model and  latent space model. Letting $\Sigma_{(i,j),(k,l)}$ be the covariance between $\mathcal{X}_{ij}$ and $\mathcal{X}_{kl}$, a general structure can be imposed for $\Sigma$:
\begin{eqnarray} \label{network:model}
\Sigma_{(i,j),(k,l)}=
\begin{cases}
1 & \{i,j\}\cap \{k,l\}= \{i,j\}\cup \{k,l\}, \\
b^* &  \{i,j\}\cap \{k,l\}=\emptyset, \\
a^* &  \mbox{otherwise},
\end{cases}
\end{eqnarray}
where $\emptyset$ is the empty set. In this case, the dependencies among edges that share common nodes and those which do not are characterized by two different parameters. The covariance structure with $b^*=0$ has been considered in relational data modeling works \citep{warner1979new, gill2001statistical, westveld2011mixed}. Here \eqref{network:model} generalizes further to take into account dependencies between edges without common nodes.

Under the model setup \eqref{model:formula}, a fundamental question is regarding the estimation of parameters in $\Sigma$, which provide important dependency information for the binary observations $A$. Given that classical asymptotic results do not hold in the current case, a delicate study of the estimation problem is not only theoretically appealing, but helpful for practitioners to better model and analyze dependent binary data sets. To fix  idea,  we consider a compound symmetry covariance structure:
\begin{align}
\label{compund:symmetry}
\Sigma_{ii}=1, ~~\Sigma_{ij}=a^*, \quad i=1,\ldots, n, \quad j=1,\ldots, n,\quad i \neq j. 
\end{align}
The correlation parameter $a^* \in (0,1)$ characterizes the dependency in the binary sequence $\{A_i\}_{i=1}^n$. This covariance matrix \eqref{compund:symmetry} is a special case of the spiked covariance form proposed in \cite{johnstone2004sparse, johnstone2009consistency} for studying high-dimensional principal component analysis. A full understanding of this structure can be seen as a gateway for understanding more complicated ones. Note that \eqref{network:model} contains \eqref{compund:symmetry} as a subset. The central question we focus on is
\begin{center}
\emph{As $n \rightarrow \infty$, can $a^*$ be consistently estimated from the sequence $\{A_i\}_{i=1}^n$?}
\end{center}
Surprisingly, the answer turns out to be negative.  We present in details our theoretical discovery in Section \ref{main:conclusion}. In particular, we analyze the likelihood function and reveal its infeasibility of producing consistent estimates. We then formally prove the nonestimability of $a^*$ under the binary sequence model \eqref{model:formula}. Interestingly, we further demonstrate that $\sqrt{n}$-consistent estimates emerge under a variant of \eqref{model:formula}. Section \ref{numerical} presents numerical experiments to support our theoretical findings. Section \ref{discussion} contains a discussion of insightful implications and future research in light of the main results in the paper. To improve readability, we put all the technical materials in Section \ref{proof:part:one} and appendix. 

\section{On the estimation of  correlation}
\label{main:conclusion}

\subsection{Model formulation} \label{model:formal:form}

Throughout the paper, we study a wide range of binary sequence models specified in \eqref{model:formula} and \eqref{compund:symmetry}. Towards that goal, we consider an equivalent model formulation in the following,
\begin{align}\label{model:formula2}
&A_i=
\begin{cases}
1 & \mbox{~if~~}\sqrt{1-a^*}Y_i+\sqrt{a^*}Y>\tau, \\
0 & \mbox{~otherwise},
\end{cases}
\nonumber \\
&\mathbb{E}(Y_i)=\mathbb{E}(Y)=0,~~\mbox{var}(Y_i)=\mbox{var}(Y)=1, \quad i=1,\ldots, n,
\end{align}
where $Y_1,\ldots, Y_n$ are i.i.d. with probability density function $p(\cdot)$ and cumulative distribution function $G(\cdot)$, independent from $Y$ with probability density function $\gamma(\cdot)$. For simplicity, we assume both $p(\cdot)$ and $\gamma(\cdot)$ are supported on $\mathbb{R}$. The latent variable $Y$ can be considered as a random effect that induces dependency in the data. Thus formulated, different expressions of $p(\cdot)$ and $\gamma(\cdot)$ form various one-parameter models. Our analysis will focus on the models satisfying some of the following regularity conditions\footnote{The density functions $p(\cdot)$ and $\gamma(\cdot)$ are not uniquely defined over a zero-measure set. The conditions are only necessary for one specification of the densities. }: 
\begin{itemize}
\item[(\textbf{A})] $\underset{z\in \mathbb{R}}{\sup}~|\frac{d \gamma(z)}{dz}|<\infty$.
\item[(\textbf{B})] $\underset{{z\in\mathbb{R}} }{\sup}~|\frac{d^3\log G(z)}{dz^3}|<\infty,\quad \underset{z\in\mathbb{R}}{\sup}~ |\frac{d^3\log(1- G(z))}{dz^3}|<\infty$.
\item[\textbf{(C)}] $\gamma(\cdot)$ is continuous, and $\underset{{|z|\rightarrow \infty}}{\limsup} \frac{\gamma(b_1z+c_1)}{\gamma(b_2z+c_2)}<\infty, \forall~ 0<b_2<b_1<\infty, c_1,c_2\in \mathbb{R}$.
\end{itemize}
The above conditions are fairly mild and satisfied by a variety of continuous distributions. We present several examples below. 

\vspace{0.2cm}
\noindent Distributions satisfying Conditions \textbf{A} and \textbf{C} include
\vspace{0.1cm}
\begin{enumerate}
\item[(1)] Standard normal: $\gamma(z)=\frac{1}{\sqrt{2\pi}}e^{-\frac{z^2}{2}}$,
\item[(2)] Scaled $t$-distribution: $\gamma(z)=\frac{\Gamma[(\nu+1)/2]}{\sqrt{(\nu-2)\pi}\Gamma(\nu/2)}(1+\frac{z^2}{\nu-2})^{-\frac{\nu+1}{2}}~(\nu>2)$,
\item[(3)] Gumbel distribution:  $\gamma(z)=\frac{\pi}{\sqrt{6}} \exp(-\frac{\pi}{\sqrt{6}} z-\gamma_0)-e^{-\frac{\pi}{\sqrt{6}} z-\gamma_0})$, where $\gamma_0$ is the Euler's constant.
\end{enumerate}

\vspace{0.1cm}

\noindent Distributions satisfying Condition \textbf{B} include
\vspace{0.1cm}
\begin{enumerate}
\item[(1)] Standard normal: $p(z)=\frac{1}{\sqrt{2\pi}}e^{-\frac{z^2}{2}}$,
\item[(2)] Logistic distribution: $p(z)=\frac{\pi e^{-\frac{\pi z}{\sqrt{3}}}}{\sqrt{3}(1+e^{-{\pi z}/{\sqrt{3}}})^2}$,
\item[(3)] Laplace distribution: $p(z)=\frac{1}{\sqrt{2}}e^{-\sqrt{2}|z|}$.
\end{enumerate}

\subsection{Likelihood analysis} \label{likelihood}
Given model \eqref{model:formula2}, our goal is to estimate the single parameter $a^*\in (0,1)$. We first write down the likelihood function 
\begin{align}
\label{likelihood:form}
L_n(a;\{A_i\}) 
=\int_{-\infty}^{\infty}\Big[1-G\Big(\frac{\tau-\sqrt{a}z}{\sqrt{1-a}}\Big)\Big]^{\sum_{i=1}^nA_i} \Big[G\Big(\frac{\tau-\sqrt{a}z}{\sqrt{1-a}}\Big)\Big]^{n-\sum_{i=1}^nA_i}\gamma(z)dz. 
\end{align}
Intuitively speaking, if maximizing the likelihood provides a consistent estimate, we should expect that a properly normalized likelihood function  converges to a deterministic function whose maximum is achieved at $a=a^*$. However, the following result says this is not the case. 

\begin{proposition} \label{prop:one}
Under the model formulation \eqref{model:formula2} and Conditions \textbf{A} and \textbf{B}, for any given $a\in (0,1)$, with probability 1, as $n\rightarrow \infty$
\begin{align}\label{normal:likelihood}
&n^{-1}\log L_n(a;\{A_i\})  \nonumber \\
\rightarrow &\bigg(1-G\Big(\frac{\tau-\sqrt{a^*}Y}{\sqrt{1-a^*}}\Big)\bigg)\cdot \log \bigg(1-G\Big(\frac{\tau-\sqrt{a^*}Y}{\sqrt{1-a^*}}\Big)\bigg) + G\Big(\frac{\tau-\sqrt{a^*}Y}{\sqrt{1-a^*}}\Big)\cdot \log G\Big(\frac{\tau-\sqrt{a^*}Y}{\sqrt{1-a^*}}\Big).
\end{align}
\end{proposition}

 Proposition \ref{prop:one} is essentially a first-order Laplace approximation. Conditions \textbf{A} and \textbf{B} guarantee the accuracy of Taylor expansions used in the Laplace method. The result reveals that the normalized log-likelihood function does not converge to a deterministic function. More interestingly, the limiting function that it converges to is invariant of $a$ for $a\in (0,1)$. The left plot in Figure \ref{fig:one} shows one example of the normalized log-likelihood function curves (the dotted lines). We observe that although the log-likelihood functions vary from sample to sample, they are all flat over $(0,1)$. This is consistent with the result in Proposition \ref{prop:one}. It is also interesting to see that the functions have a sharp transition around $a=0$. This can be justified by a straightforward calculation that shows $n^{-1}\log L_n(0;\{A_i\})\rightarrow -\log2\approx -0.693$. Hence the limit at $a=0$ is deterministic and does not even depend on $a^*$. 
 
 It is possible to prove that the convergence in \eqref{normal:likelihood} continues to hold in expectation, under possibly additional conditions. To keep the technical material concise, we thus do not present a formal proof in the paper. However, we have numerically verified this result. The solid curve in the left plot of Figure \ref{fig:one} shows one specific example of normalized expected log-likelihood function, and the dashed line corresponds to the expectation of the limit in \eqref{normal:likelihood}. The two curves are both flat over $(0,1)$ and well aligned with each other. More numerical results are available in Section \ref{numerical}.

\begin{figure}[htb!]
\centering
\begin{tabular}{cc}
\hspace{-0.7cm} \includegraphics[width=8cm, height=8.5cm]{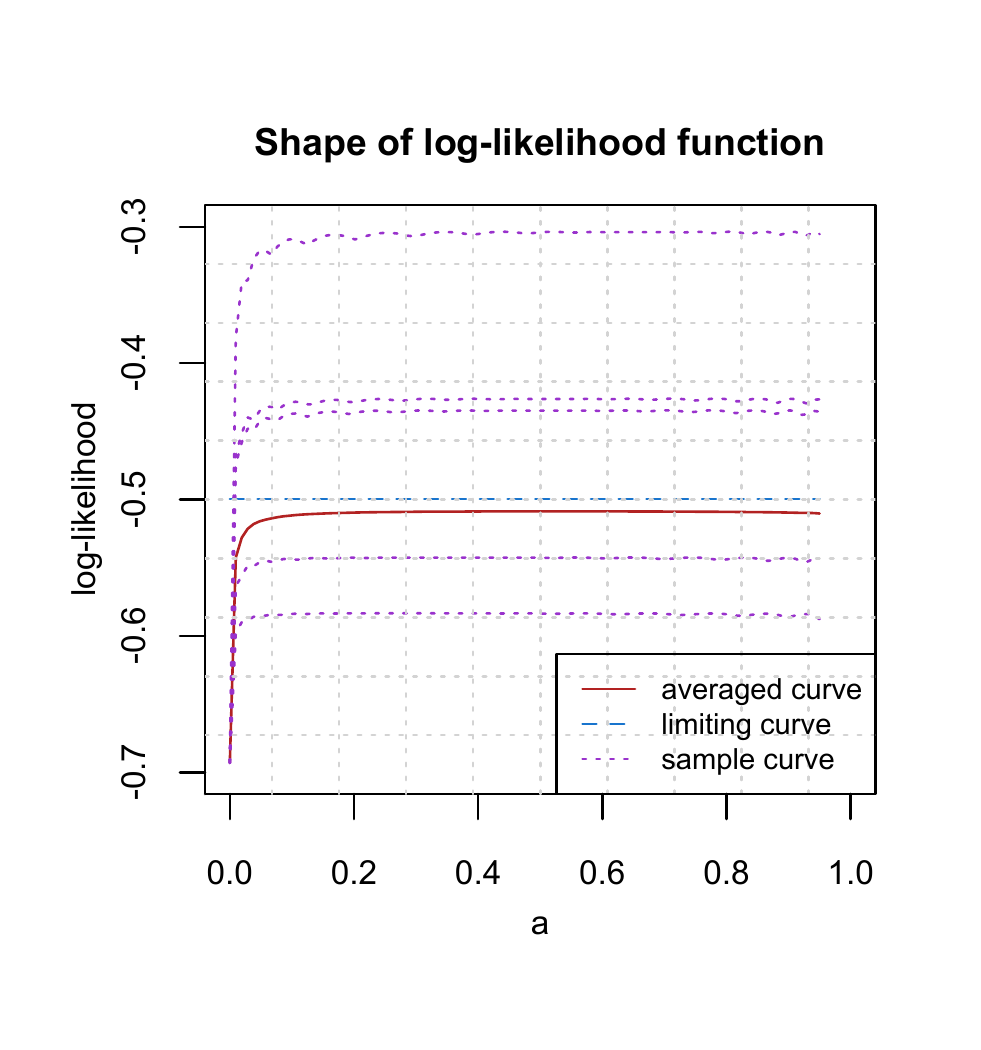} &\hspace{-0.7cm}
 \includegraphics[width=8cm, height=8.5cm]{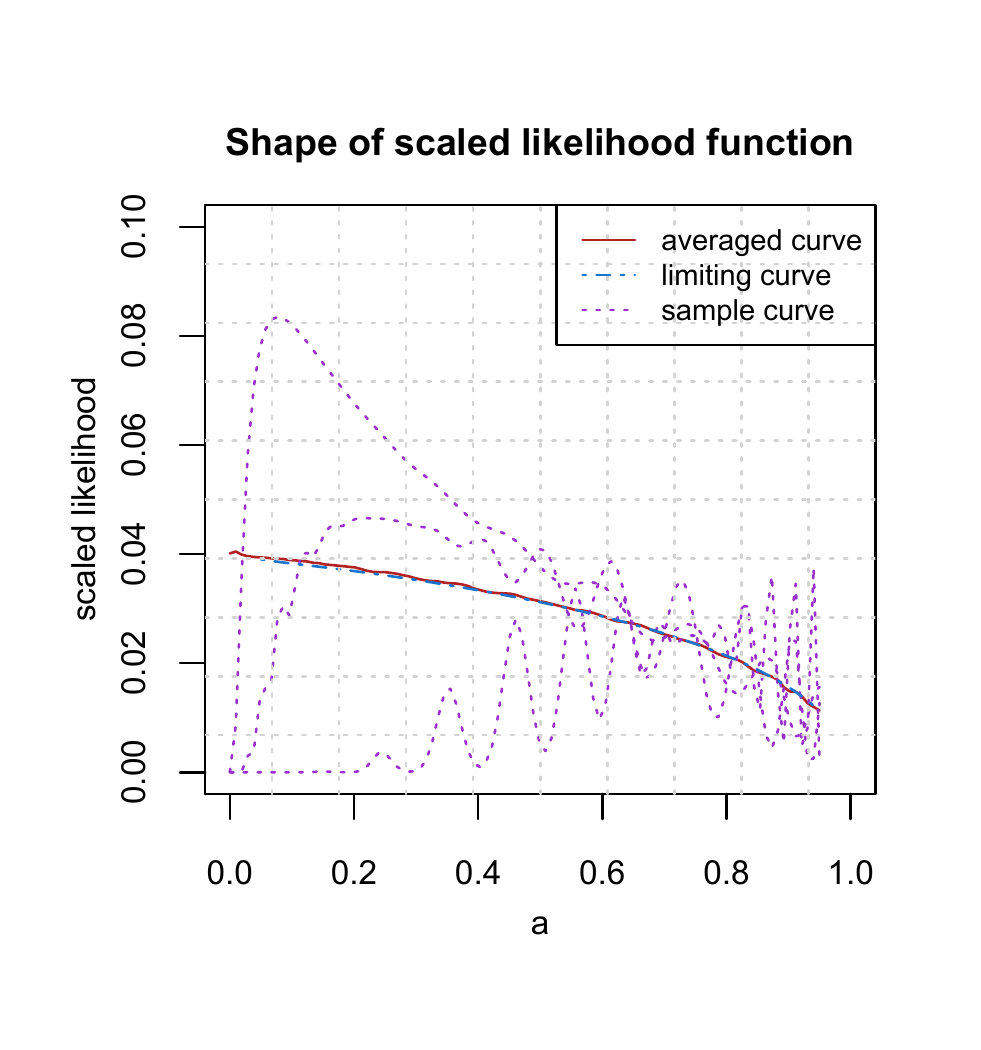} 
\end{tabular}
\vspace{-0.7cm}
\caption{Plots of normalized log-likelihood (left) and scaled likelihood (right) under model \eqref{model:formula2}, where $\alpha^*=0.5, p(z)=\gamma(z)=\frac{1}{\sqrt{2\pi}}e^{-\frac{z^2}{2}}, \tau=0, n=1000$; ``averaged curve" denotes the expected functions; ``limiting curve" is the expectation of the limit in \eqref{normal:likelihood} and \eqref{second:order:term}; ``sample curve" represents the functions computed from one sample.} \label{fig:one}
\end{figure}

While the normalized (expected) log-likelihood function converges (in first order) to an unfavorable object, it does not necessarily mean that likelihood-based estimates are not consistent, because the higher-order terms in the limiting function may contain useful information for the parameter of interest. Unfortunately, this is not true either, as demonstrated below by a second-order analysis of the likelihood function. The proof of the preceding and following propositions are relegated to the appendix.

\begin{proposition}\label{higherorder:like}

Under model formulation \eqref{model:formula2} and Conditions \textbf{A} and \textbf{B}, for any given $a\in (0,1)$, as $n\rightarrow \infty$,
\begin{align}
\label{second:order:term}
& L_n(a;\{A_i\})\cdot \exp\left(-n[\bar{A}\log\bar{A}+(1-\bar{A})\log(1-\bar{A})]\right) \nonumber \\
=&\frac{\sqrt{1-a}\gamma\Big(\sqrt{\frac{(1-a)a^*}{(1-a^*)a}}Y+\frac{\sqrt{1-a^*}-\sqrt{1-a}}{\sqrt{a(1-a^*)}}\tau\Big)}{\sqrt{a}}\cdot \sqrt{\frac{2\pi\left[1-G(\frac{\tau-\sqrt{a^*}Y}{\sqrt{1-a^*}})\right]\cdot G(\frac{\tau-\sqrt{a^*}Y}{\sqrt{1-a^*}})}{np^2(\frac{\tau-\sqrt{a^*}Y}{\sqrt{1-a^*}})}}+O_p(n^{-1}),
\end{align}
where $\bar{A}=\frac{1}{n}\sum_{i=1}^nA_i$.
\end{proposition}

Proposition \ref{higherorder:like} can be considered as a second-order Laplace approximation result. Conditions \textbf{A} and \textbf{B} are used to justify the Taylor expansions in the analysis. 
We have shown in Proposition \ref{prop:one} that the first-order term of the likelihood function $L_n(a;\{A_i\})$ is exponentially small. We push the analysis one step further in Proposition \ref{higherorder:like} to reveal that the second-order dominating term in $L_n(a;\{A_i\})$ is of order $n^{-1/2}$, and derive the precise constant of the second-order term. 

In contrast to the first-order term in Proposition \ref{prop:one}, the second-order term does depend on $a$. However, this term is a random function of $a$, and in general does not attain maximum at the true parameter $a=a^*$. For example, in the case where $\gamma(z)=p(z)=\frac{1}{\sqrt{2\pi}}e^{-\frac{z^2}{2}}$, the maximizer of this term is $a=\frac{a^*Y^2}{a^*Y^2+1-a^*}$. As in Proposition \ref{prop:one}, the convergence in \eqref{second:order:term} can be proved to hold in expectation under more regularity conditions. We skip this part for simplicity. 

The right plot of Figure \ref{fig:one} illustrates the result in Proposition \ref{higherorder:like}. We see that the scaled likelihood functions (dotted curves) depend on $a$ in a nonlinear way, but they are not able to identify the truth parameter $a^*$ as the (approximate) maximizer. Moreover, the expected likelihood (solid curve) is well matched with the dashed curve, the expectation of the limit  in \eqref{higherorder:like}. We provide more numerical examples in Section \ref{numerical}.

\subsection{Nonestimability of $a^*$}\label{estimability}

The analysis of the likelihood function in Section \ref{likelihood} indicates that likelihood-based methods may not yield consistent estimates for $a^*$. 
Given the optimality of maximum likelihood estimation under parametric models, it may further imply that no consistent estimates ever exist. Indeed, this is formally established in the theorem below.

\begin{theorem}\label{minimax:lowerb}
Under model \eqref{model:formula2} and Condition \textbf{C}, given any strictly increasing function $R(\cdot): [0,\infty) \rightarrow [0,\infty)$ with $R(0)=0$ and any $0<a_1<a_2<1$,
\begin{eqnarray*}
\liminf_{n\rightarrow \infty} \inf_{T}\max_{a^*\in \{a_1,a_2\}}E_{a^*}[R(|T(\{A_i\}_{i=1}^n)-a^*|)]>0,
\end{eqnarray*}
where $T$ is any measurable function and the expectation $E_{a^*}$ is taken over $\{A_i\}_{i=1}^n$.
\end{theorem}

The proof is presented in Section \ref{proof:thm1}.

\begin{remark} 
\label{remark:one}
 Theorem \ref{minimax:lowerb} reveals that no consistent estimates of $a^*$ exist even when the parameter space only consists of two distinct elements. At first glance, this result seems counter-intuitive. The model \eqref{model:formula2} has a simple structure with a single parameter, but there is no way to reliably estimate the parameter from infinite components of the data sequence. On the other hand, suppose instead of the thresholded sequence $(A_1,\ldots, A_n)$, we observe the hidden sequence $(X_1,\ldots, X_n)$ where $X_i=\sqrt{1-a^*}Y_i+\sqrt{a^*}Y$. Then $\sqrt{n}$-consistent estimates can be readily available. For instance, denote $Z_i=X_{2i-1}-X_{2i},  i=1,\ldots,  \lfloor n/2\rfloor$. Then it is clear that $Z_1,\ldots, Z_{n/2}$ are independently and identically distributed with $\mathbb{E}Z_1=0, \mbox{var}(Z_1)=2(1-a^*)$. Hence $1-n^{-1}\sum_{i=1}^{n/2}Z^2_i$ is $\sqrt{n}$-consistent for $a^*$. 
 \end{remark}

\begin{remark}
 One might argue that it is the strong dependency in $\{A_i\}_{i=1}^n$ causing the issue. If the dependency can be somehow weakened, will $a^*$ become estimable? For example, we can consider $a^*=\rho_na^*_0$ with $a^*_0$ being a constant and $\rho_n\rightarrow 0$, and ask if there is any estimate $\hat{a}_n$ such that $\hat{a}_n-a^*=o(\rho_n)$. With a similar proof as the one for Theorem \ref{minimax:lowerb}, it is possible to show
 \[
\liminf_{n\rightarrow \infty} \inf_{T}\max_{a_0^*\in \{a_1,a_2\}}E_{a^*}(\rho^{-1}_n|T(\{A_i\}_{i=1}^n)-a^*|)>0. 
\]
Hence $a^*$ remains nonestimable. 
\end{remark}

\subsection{From binary to trinary data} \label{solution}

The discussion in Remark \ref{remark:one} from Section \ref{estimability} provides an intriguing explanation for the nonestimability of $a^*$ under model \eqref{model:formula2}. Specifically, $a^*$ is estimable from $(X_1,\ldots, X_n)$, but becomes nonestimable from the thresholded sequence $(A_1,\ldots, A_n)$. It is the thresholding operation in the binary data generating process that results in loss of too much information to consistently estimate $a^*$. In a nutshell, one-bit information of each $X_i$ is not sufficient to recover the common correlation among them. Then a natural question arises as how much information of each $X_i$ suffices to estimate $a^*$. We give a precise answer to the question in this section. Interestingly, we will show that in fact a little more than one bit of information is sufficient to obtain $\sqrt{n}$-consistent estimates.

Given two constants $-\infty<\tau_1 < \tau_2<+\infty$, let $\mathcal{I}_1=(-\infty,\tau_1], \mathcal{I}_2=(\tau_1,\tau_2], \mathcal{I}_3=(\tau_2,\infty)$ be the three adjoining intervals separated by the two break points. Suppose we observe the trinary sequence $(A_1,\ldots, A_n)$ where for each $i=1,\ldots,n,$
\begin{align}
\label{trinary:model}
&A_i=(A^{(1)}_i, A^{(2)}_i, A^{(3)}_i), ~A^{(j)}_i={1}\{X_i \in \mathcal{I}_j\}, ~j=1,2,3,  \nonumber \\
&X_i=\sqrt{1-a^*}Y_i+\sqrt{a^*}Y,
\end{align}
in which the $Y_i$'s and $Y$ are the same as in model \eqref{model:formula2}. Note that when $\tau_1=\tau_2=\tau$, model \eqref{trinary:model} is reduced to the binary sequence \eqref{model:formula2}. The theorem below gives one $\sqrt{n}$-consistent estimate for $a^*$ under the trinary sequence model \eqref{trinary:model}.

\begin{theorem}\label{three:thm}
Denote
\[
\bar{A}^{(j)}=n^{-1}\sum_{i=1}^nA_i^{(j)}, \quad j=1,2, 3.
\]
For any given $\tau_1 <\tau_2$, consider the estimate 
\begin{eqnarray}\label{moment:estimate}
\hat{a}_n=1-\Big[\frac{\tau_1-\tau_2}{G^{-1}(\bar{A}^{(1)})-G^{-1}(1-\bar{A}^{(3)})}\Big]^2\cdot \mathbbm{1}_{\bar{A}^{(1)}>0, \bar{A}^{(2)}>0,\bar{A}^{(3)}>0},
\end{eqnarray}
where $G^{-1}(\cdot)$ is the inverse function of the cumulative distribution function $G(\cdot)$. Under model \eqref{trinary:model}, it holds that
\[
\sqrt{n}(\hat{a}_n-a^*)=O_p(1), \quad \mbox{~~as~}n\rightarrow \infty.
\]
\end{theorem}

The proof of Theorem \ref{three:thm} can be found in Section \ref{thm2:proof:add}.

\begin{remark}
Following the preceding discussion, we can further consider a general sequence model $(A_1,\ldots, A_n)$ where there exist $S$ ($S \geq 3$) consecutive intervals $\{ \mathcal{I}_j \}_{j=1}^S$ and $A_i$ represents which interval $X_i$ falls into. It is not hard to show that $\sqrt{n}$-consistent estimate for $a^*$ can be constructed in a similar way as in \eqref{moment:estimate}. At a high level, we may consider the observed sequence $\{A_i\}_{i=1}^n$ as a  discretized version of the latent continuous sequence $\{X_i\}_{i=1}^n$. Theorems \ref{minimax:lowerb} and \ref{three:thm} together characterize a phase transition regarding the estimability of $a^*$. That is, $a^*$ is estimable if and only if $S \geq 3$.  
\end{remark}

\section{Numerical experiments}\label{numerical}

\subsection{Simulations}

In this section, we provide some simulation studies to validate the theoretical results presented in Sections \ref{likelihood}, \ref{estimability}, and \ref{solution}. We consider the following three cases\footnote{The parameters in the distributions are chosen to make them mean zero and variance one.}:
\begin{itemize}
\item[(1)] Standard normal: $p(z)=\frac{1}{\sqrt{2\pi}}e^{-\frac{z^2}{2}}$.~~Standard normal: $\gamma(z)=\frac{1}{\sqrt{2\pi}}e^{-\frac{z^2}{2}}$.
\item[(2)] Logistic distribution: $p(z)=\frac{\pi e^{-\frac{\pi z}{\sqrt{3}}}}{\sqrt{3}(1+e^{-\frac{\pi z}{\sqrt{3}}})^2}$.~~Standard normal: $\gamma(z)=\frac{1}{\sqrt{2\pi}}e^{-\frac{z^2}{2}}$.
\item[(3)] Laplace distribution: $p(z)=\frac{1}{\sqrt{2}}e^{-\sqrt{2}|z|}$.~~Scaled $t$-distribution: $\gamma(z)=\frac{\Gamma(3)}{\sqrt{3\pi}\Gamma(2.5)}(1+\frac{z^2}{3})^{-3}$.
\end{itemize}
Throughout the simulations, we set $a^*=0.5$, $\tau=0$ in the binary sequence model \eqref{model:formula2} and $a^*=0.5$, $\tau_1=-1$, $\tau_2=1$ in the trinary sequence model \eqref{trinary:model}. 

For each case, we plot the normalized log-likelihood and scaled likelihood under model \eqref{model:formula2} with $n=1000$, and log-likelihood under model \eqref{trinary:model} with $n=500$. The expectations of the aforementioned functions are included as well. The results are shown in Figure \ref{fig:two}. Let us use Case 1 (first row) as an example to discuss the details. Similar phenomena are observed in the other two cases. As seen in the first plot, the normalized log-likelihood functions (dotted lines) are flat over $(0,1)$. Moreover, the expected log-likelihood (solid line) is also flat over $(0,1)$ and well approximated by the dashed line, the expectation of the limit in \eqref{normal:likelihood}. These outcomes are accurately predicted by Proposition \ref{prop:one}. The second plot confirms the results in Proposition \ref{higherorder:like}: the rescaled likelihoods (dotted lines), as a function of $a$, vary much from sample to sample; and they can not identify the true parameter $a^*=0.5$ as the (approximate) maximizer. In addition, the expected function (solid line) is close to the expectation of the limit derived in \eqref{second:order:term}. Finally, in the third plot, it is clear that the (expected) log-likelihood function becomes informative for the estimation of $a^*$ under the trinary sequence model. This is consistent with our discussions in Section \ref{solution}.

To evaluate the estimator $\hat{a}_n$ in Theorem \ref{three:thm}, for each of the three cases, we compute the averaged absolute error $|\hat{a}_n-a^*|$ over $5000$ repetitions, with sample sizes varying from $1000$ to $3000$. Figure \ref{fig:three} shows the error vs. sample size plot on the logarithmic scale. We observe that all three curves are nearly linear. Using least squares, the estimated slopes of them are $-0.55, -0.52, -0.52$, respectively. These numerical results verified that $\hat{a}_n$ is a $\sqrt{n}$-consistent estimator.

\begin{figure}[htb!]
\centering
\begin{tabular}{ccc}
\hspace{-.8cm}\includegraphics[width=5.7cm, height=6.2cm]{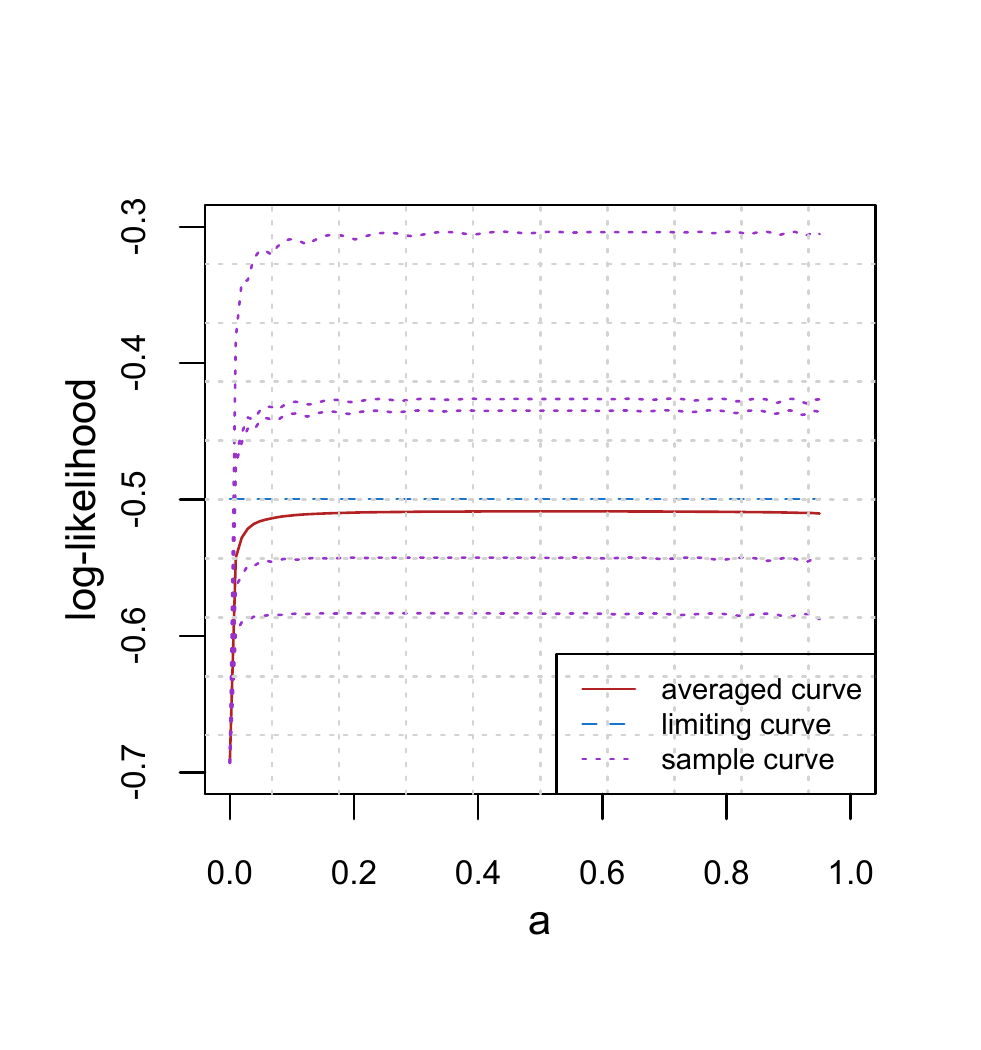} &\hspace{-1.3cm}
 \includegraphics[width=5.7cm, height=6.2cm]{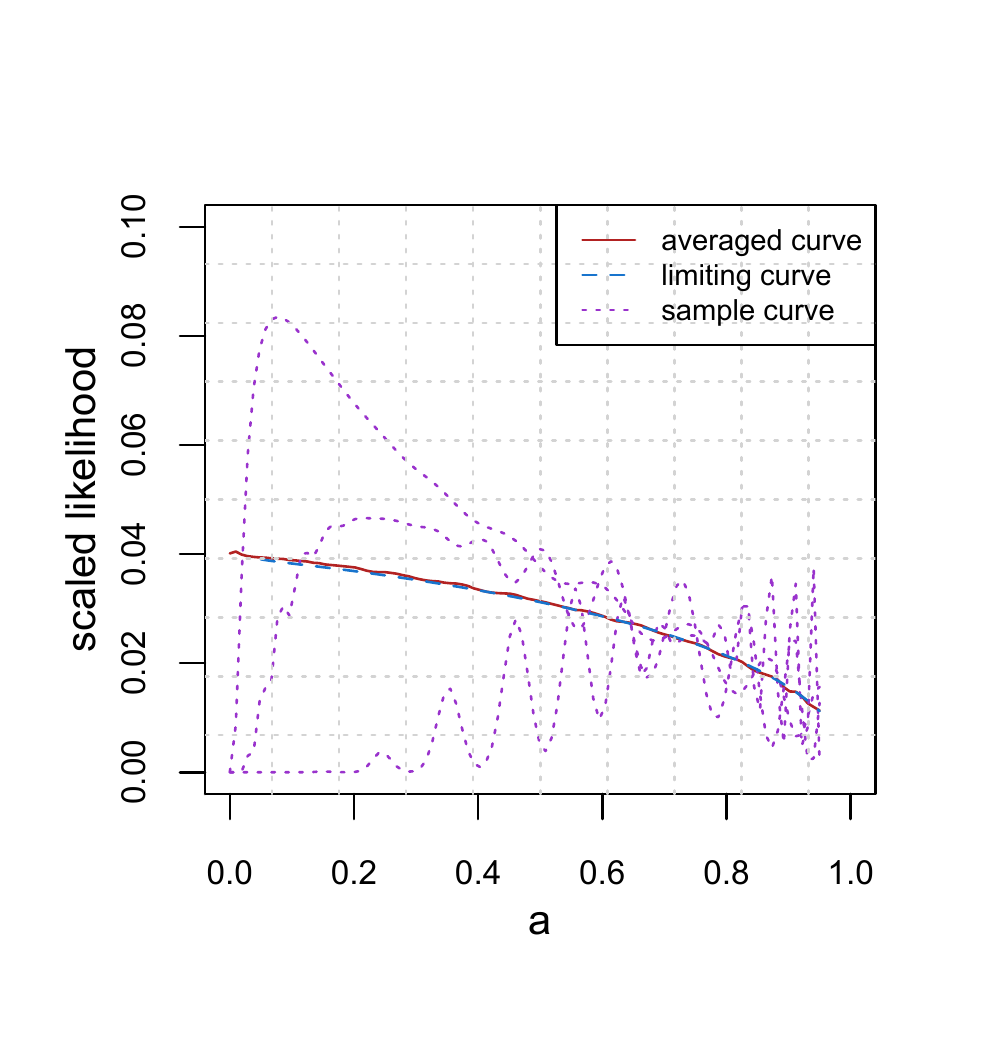} &\hspace{-1.3cm}
 \includegraphics[width=5.7cm, height=6.2cm]{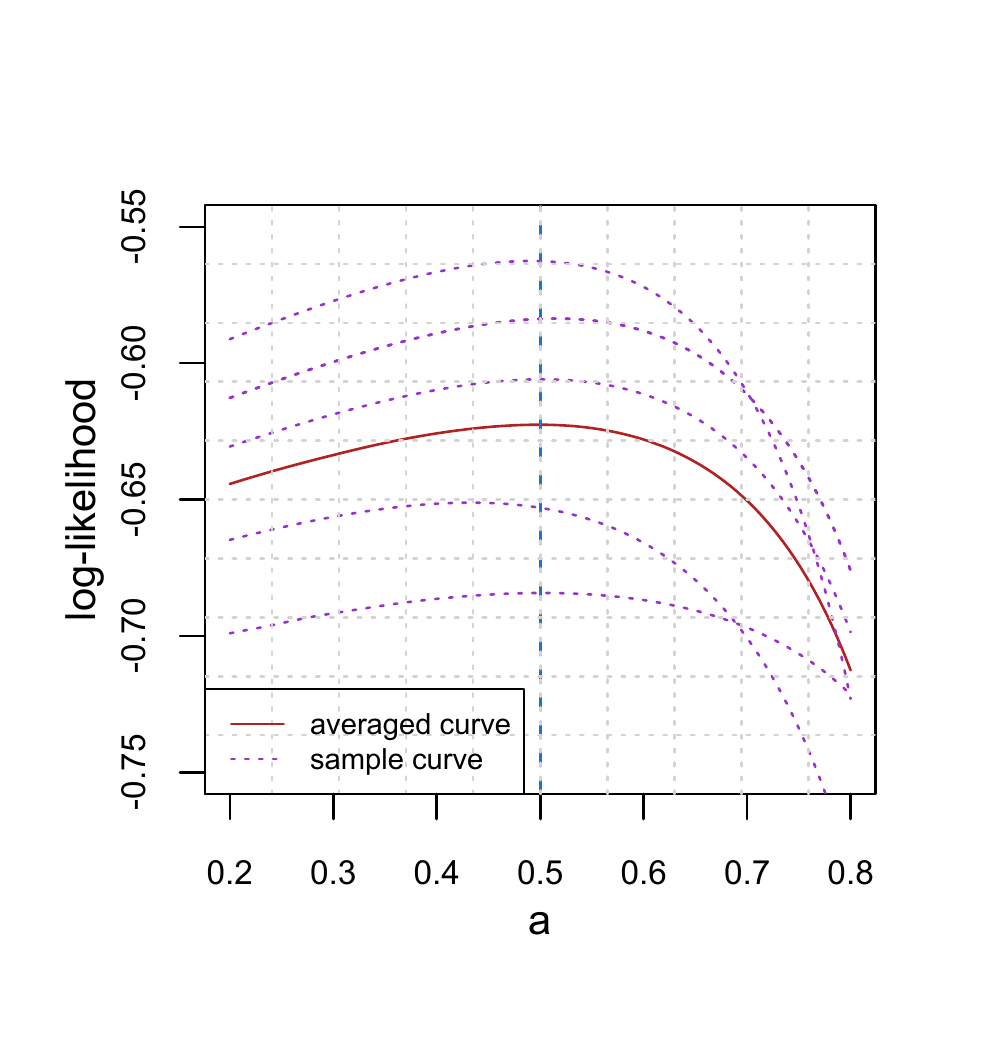}   \\ [-9ex]
\hspace{-.8cm} \includegraphics[width=5.7cm, height=6.2cm]{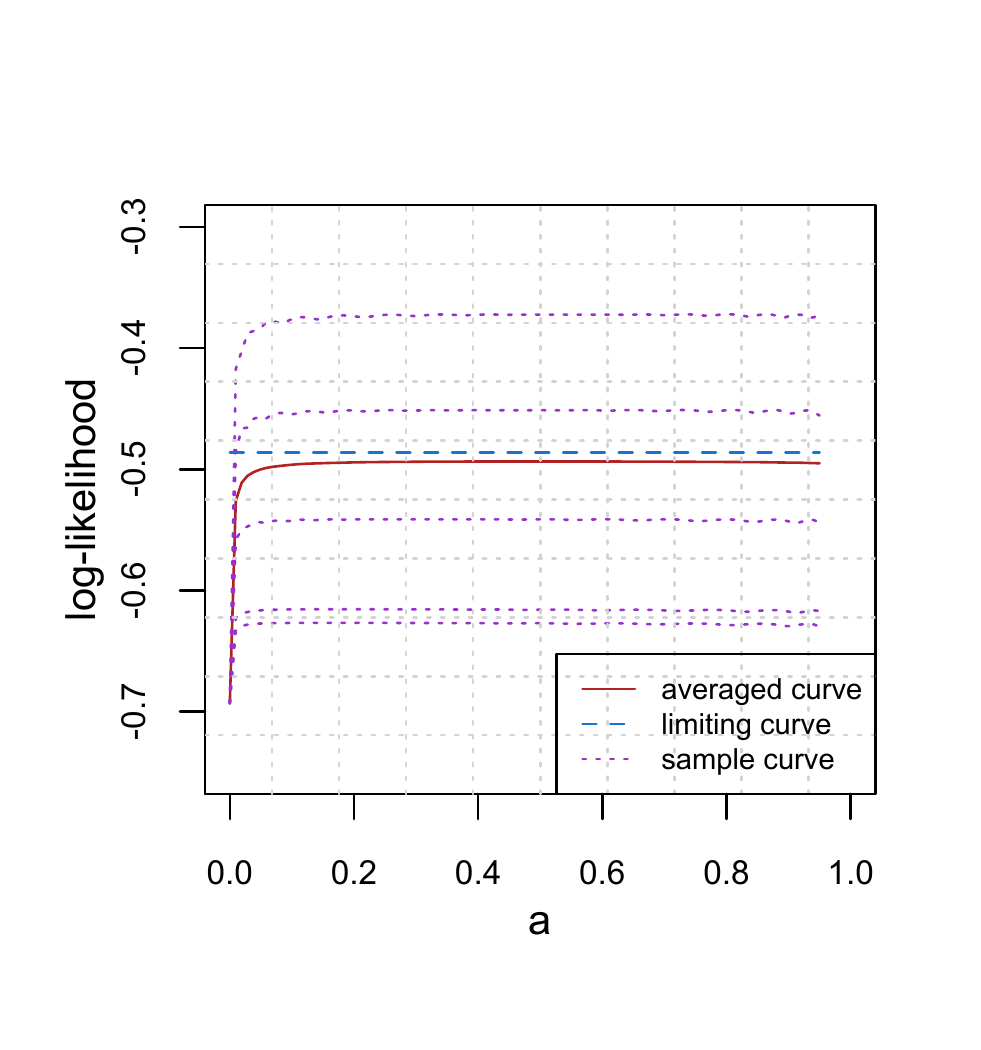} &\hspace{-1.3cm}
 \includegraphics[width=5.7cm, height=6.2cm]{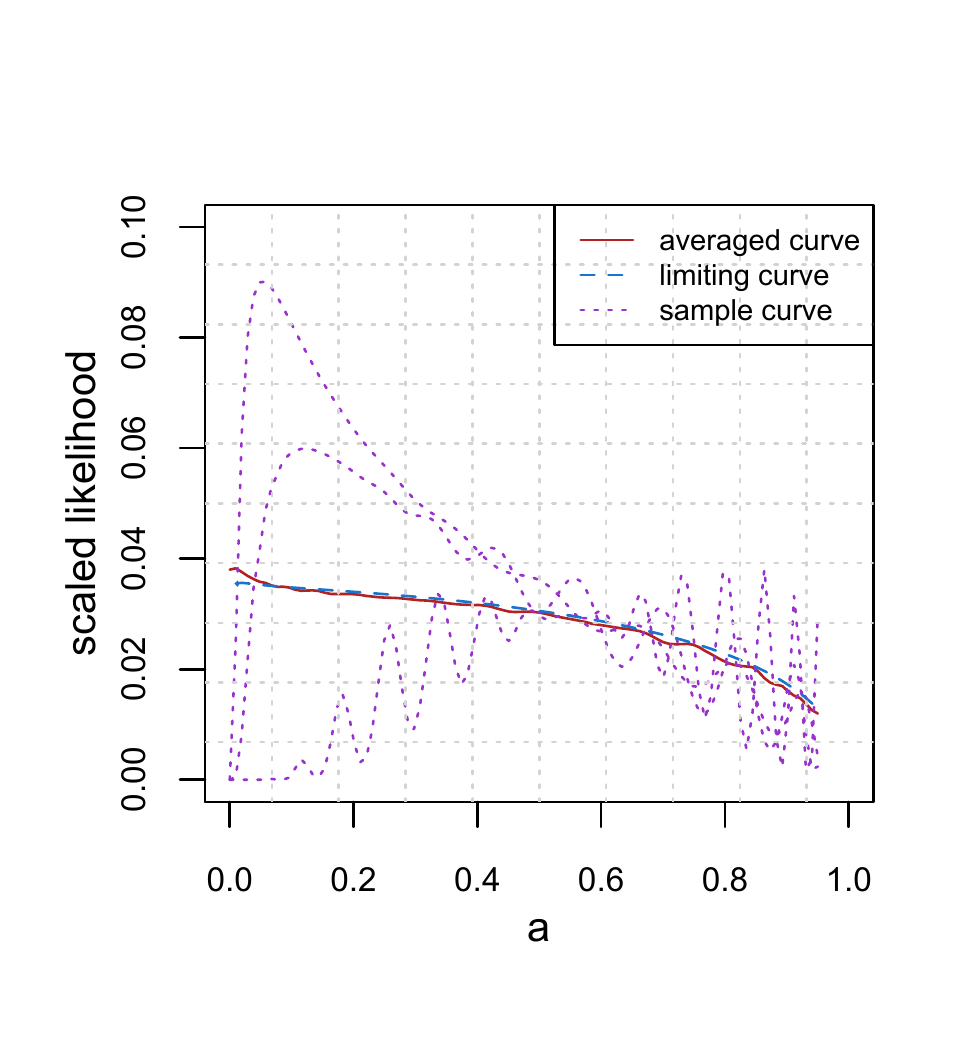} &\hspace{-1.3cm}
 \includegraphics[width=5.7cm, height=6.2cm]{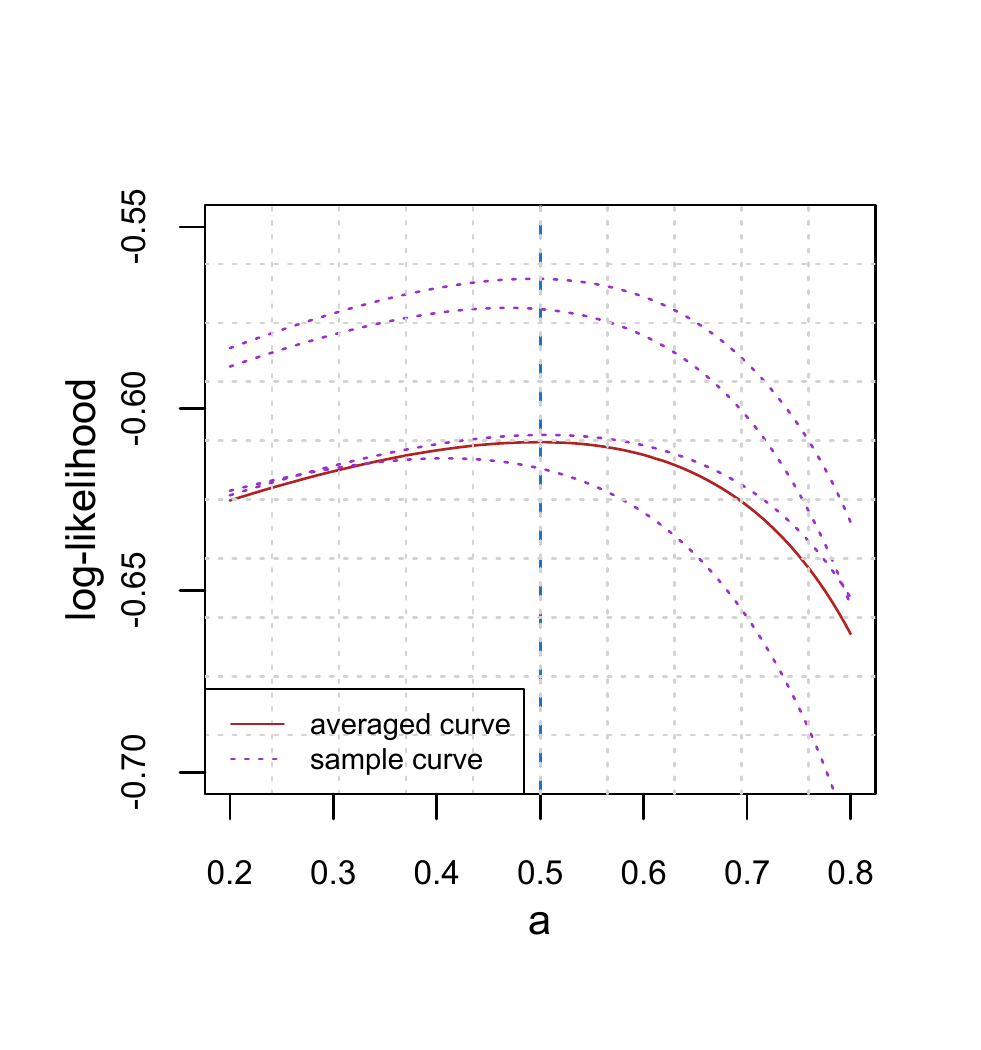}  \\[-9ex]
\hspace{-.8cm} \includegraphics[width=5.7cm, height=6.2cm]{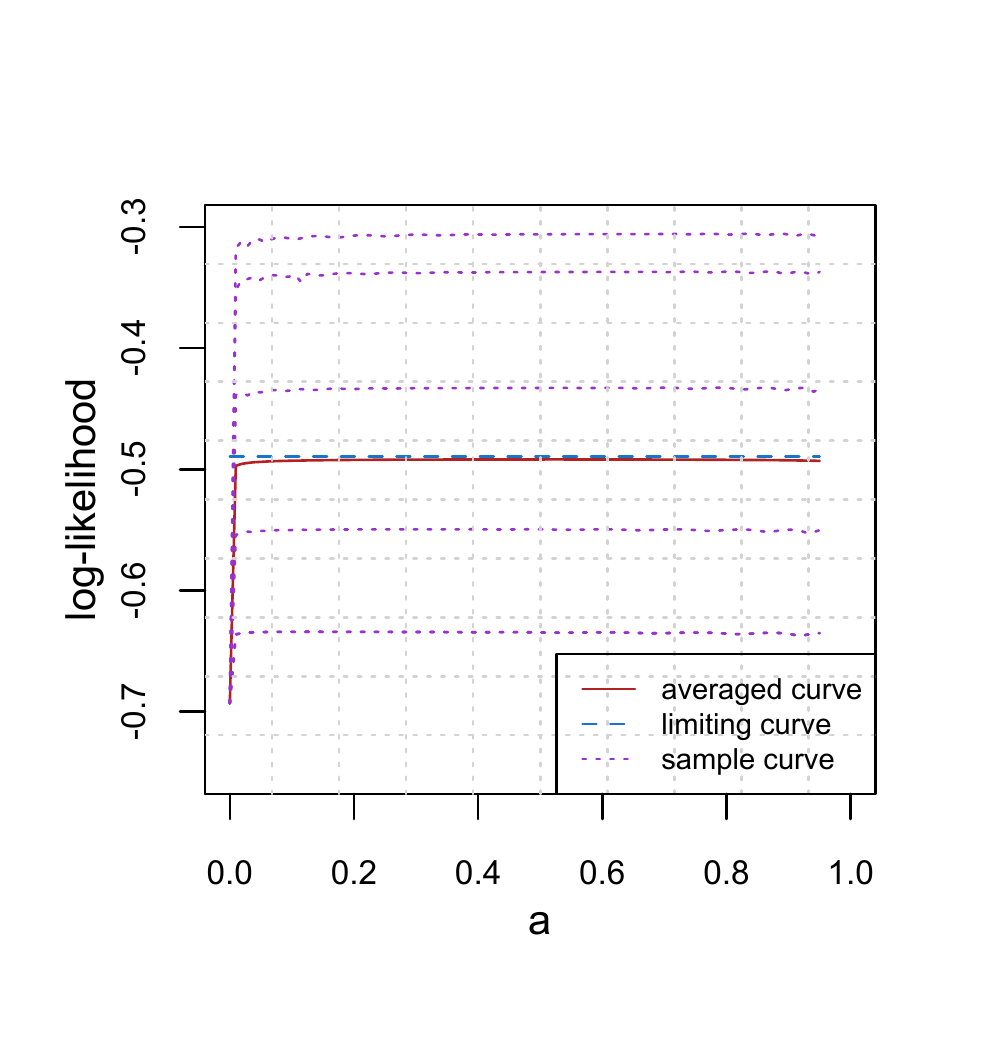} &\hspace{-1.3cm}
 \includegraphics[width=5.7cm, height=6.2cm]{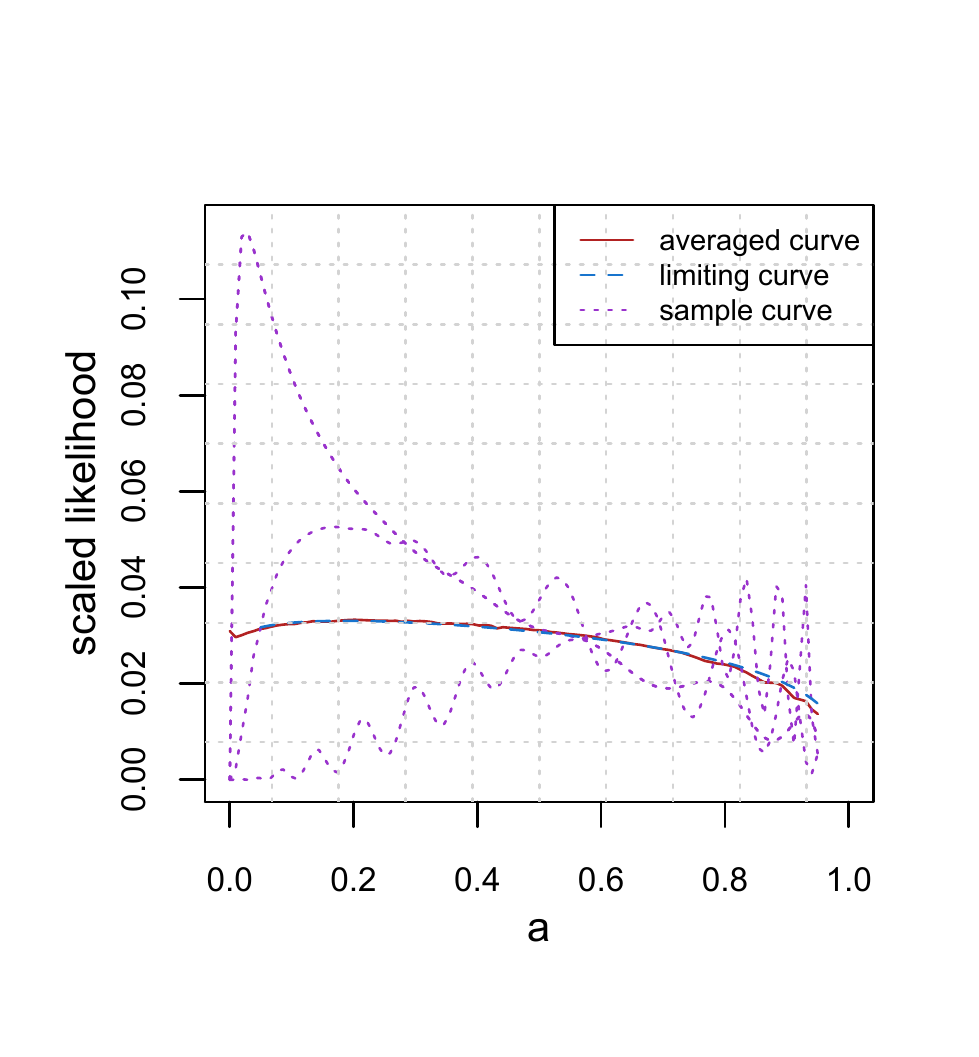} & \hspace{-1.3cm}
 \includegraphics[width=5.7cm, height=6.2cm]{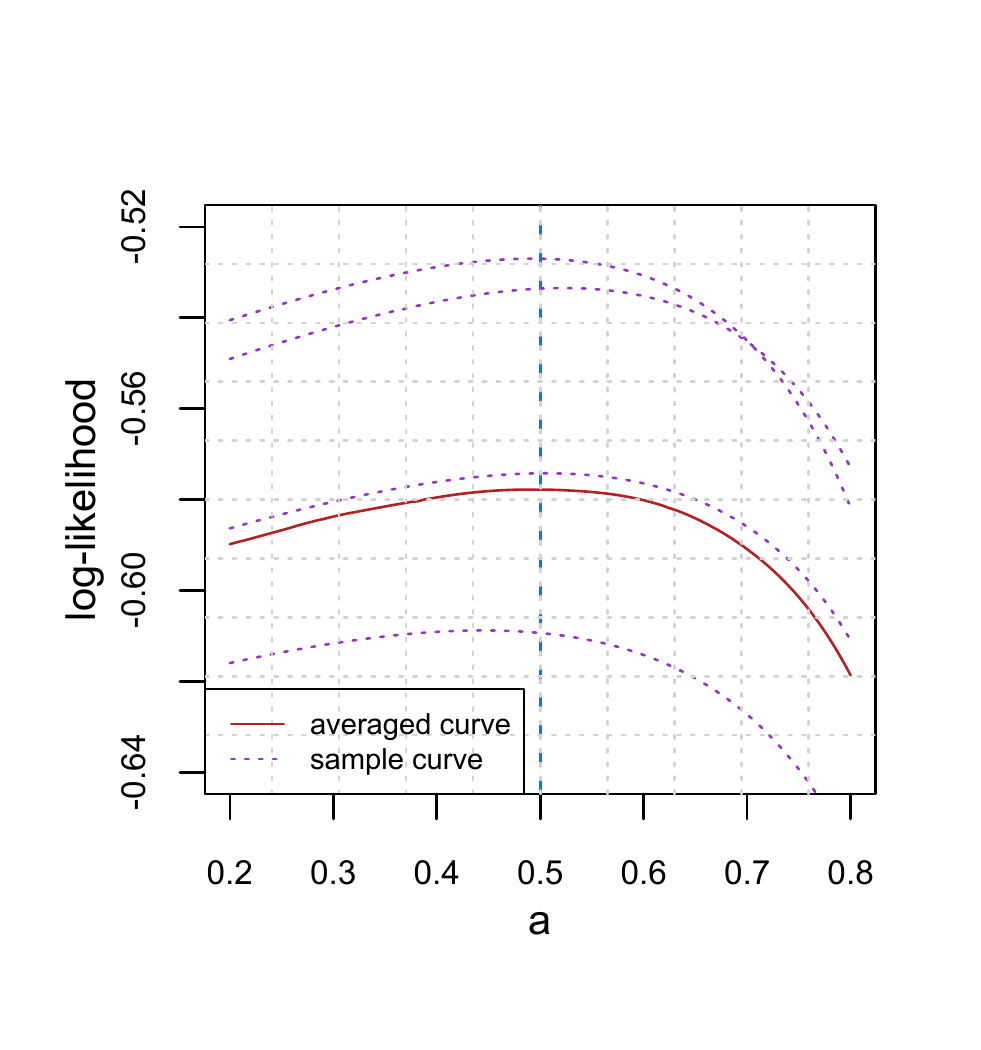}  
\end{tabular}
\vspace{-0.7cm}
\caption{Plots of normalized log-likelihood (first column) under model \eqref{model:formula2}, scaled likelihood (second column) under model \eqref{model:formula2}, and normalized log-likelihood (third column) under model \eqref{trinary:model}, for Case 1 (first row), Case 2 (second row), and Case 3 (third row).} \label{fig:two}
\end{figure}

\begin{figure}[htb]
\centering
\begin{tabular}{c}
\hspace{-0.7cm} \includegraphics[width=9cm, height=9.5cm]{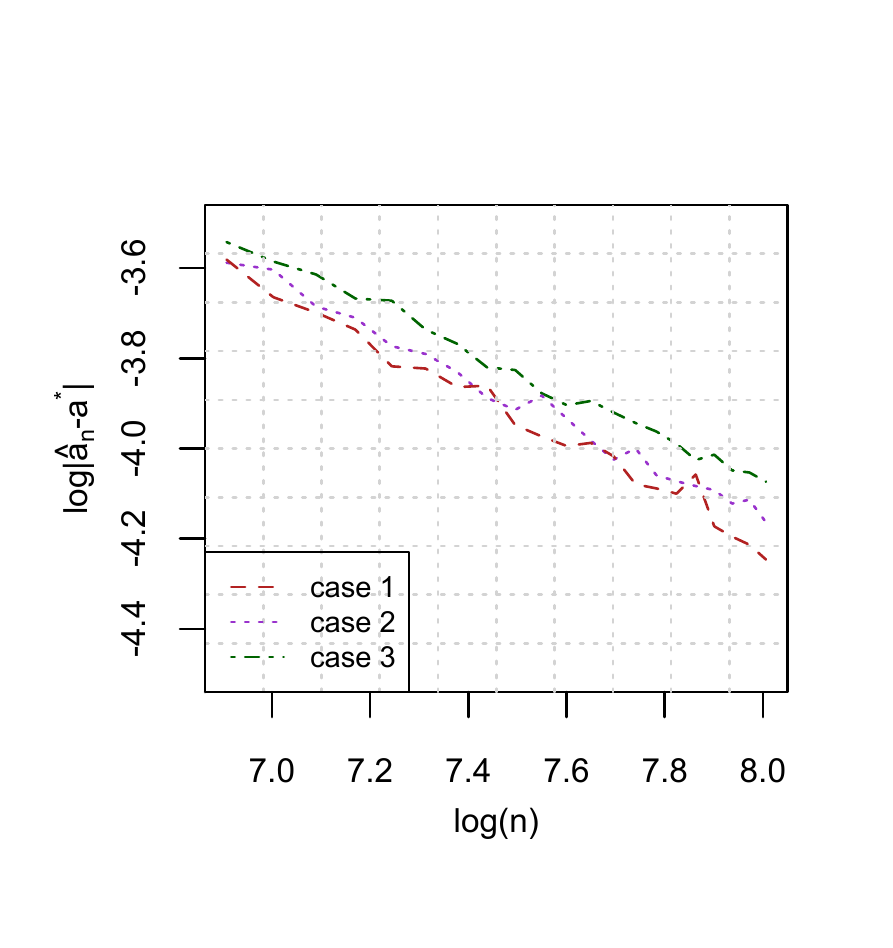} 
\end{tabular}
\vspace{-0.7cm}
\caption{Estimation error of $\hat{a}_n$ on logarithmic scale. } \label{fig:three}
\end{figure}

\subsection{Real data example}

In this section, we use stock market data to further demonstrate the estimability phenomenon of $a^*$ using binary and trinary observations. The dataset \citep{cam2018} contains 5-year (from 02-08-2013 to 02-07-2018) historical stock prices for all companies currently found on the S\&P 500 index. We merely focus on the 63 stocks from the financial sector. The quantity of interest is the logarithmic return defined as
\[
R^{(t)}_i=\log \frac{V^{(t)}_i}{V^{(t-1)}_{i}}, 
\]
where $V^{(t)}_i$ and  $V^{(t-1)}_{i}$ are the closing prices for stock $i$ on day $t$ and day $t-1$, respectively. Given a date $t$ and a stock $i$, we use $X^{(t)}_i$ to denote the standardized $R^{(t)}_i$ using the historical data $\{R^{(t-k)}_i\}_{k=1}^{100}$. The sequence $X^{(t)}=(X^{(t)}_1, \ldots, X^{(t)}_{63})$ characterizes the performances of the 63 financial stocks on day $t$. Figure \ref{fig:five} shows the Q-Q plots for $X^{(t)}$ on 6 different dates. It provides strong evidence for the Gaussianity of $X^{(t)}$. Moreover, since the 63 stocks come from the same sector, significant correlation is expected to exist among them. Therefore, we model $X^{(t)}$ as a Gaussian sequence with
\begin{align}
\label{gaussian:sequence}
X^{(t)}_i=\sqrt{1-a^*}Y_i+\sqrt{a^*}Y, \mbox{~~~where~~}Y_i's, Y\overset{i.i.d.}{\sim} N(0,1).
\end{align}
We use the following U-statistic as a reference for the ground truth $a^*$,
\begin{align}
\label{reference:estimate}
\hat{a}^{(t)}= 1-\frac{1}{2{63 \choose 2}}\sum_{1\leq i<j\leq 63}(X^{(t)}_i-X_j^{(t)})^2.
\end{align}
Note that the above estimator $\hat{a}^{(t)}$ is unbiased and $\sqrt{n}$-consistent. We obtain the binary sequence and trinary sequence by setting the thresholds $\tau=0$ and $\tau_1=-0.5, \tau_2=0.5$, respectively. We then estimate $a^*$ by maximum likelihood estimator using the binary data, and by our proposed estimator in \eqref{moment:estimate} with the trinary data. The results are summarized in Figure \ref{fig:four}. Clearly, the estimator in 	\eqref{moment:estimate} using the trinary sequence has a close performance compared with the unbiased and  consistent estimator \eqref{reference:estimate} based on the full data, while $a^*$ is poorly estimated using the binary sequence. The result delivers an interesting message: the information that the stock prices rise or fall (encoded in the binary data) is not sufficient to accurately estimate the common correlation of the stock prices; on the other hand, a bit more suffices. This is precisely the main conclusion of the paper. 
\begin{figure}[htb]
\centering
\begin{tabular}{c}
\hspace{-0.7cm} \includegraphics[width=9cm, height=9.5cm]{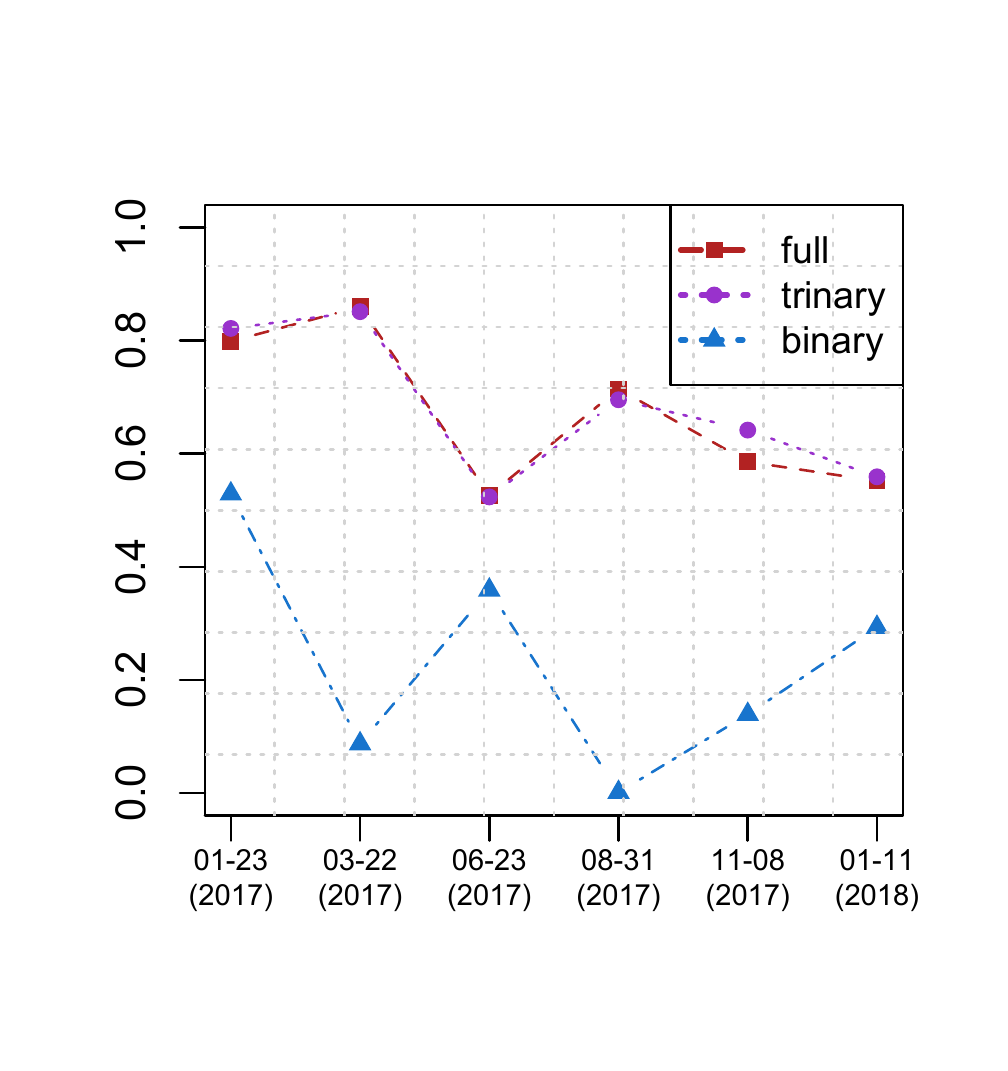} 
\end{tabular}
\vspace{-0.7cm}
\caption{Estimates for the common correlation of logarithmic returns using full data (estimator \eqref{reference:estimate}), trinary data (estimator \eqref{moment:estimate}), and binary data (MLE).} \label{fig:four}
\end{figure}

\begin{figure}[htb!]
\centering
\begin{tabular}{ccc}
\hspace{-.8cm}\includegraphics[width=5.7cm, height=6.2cm]{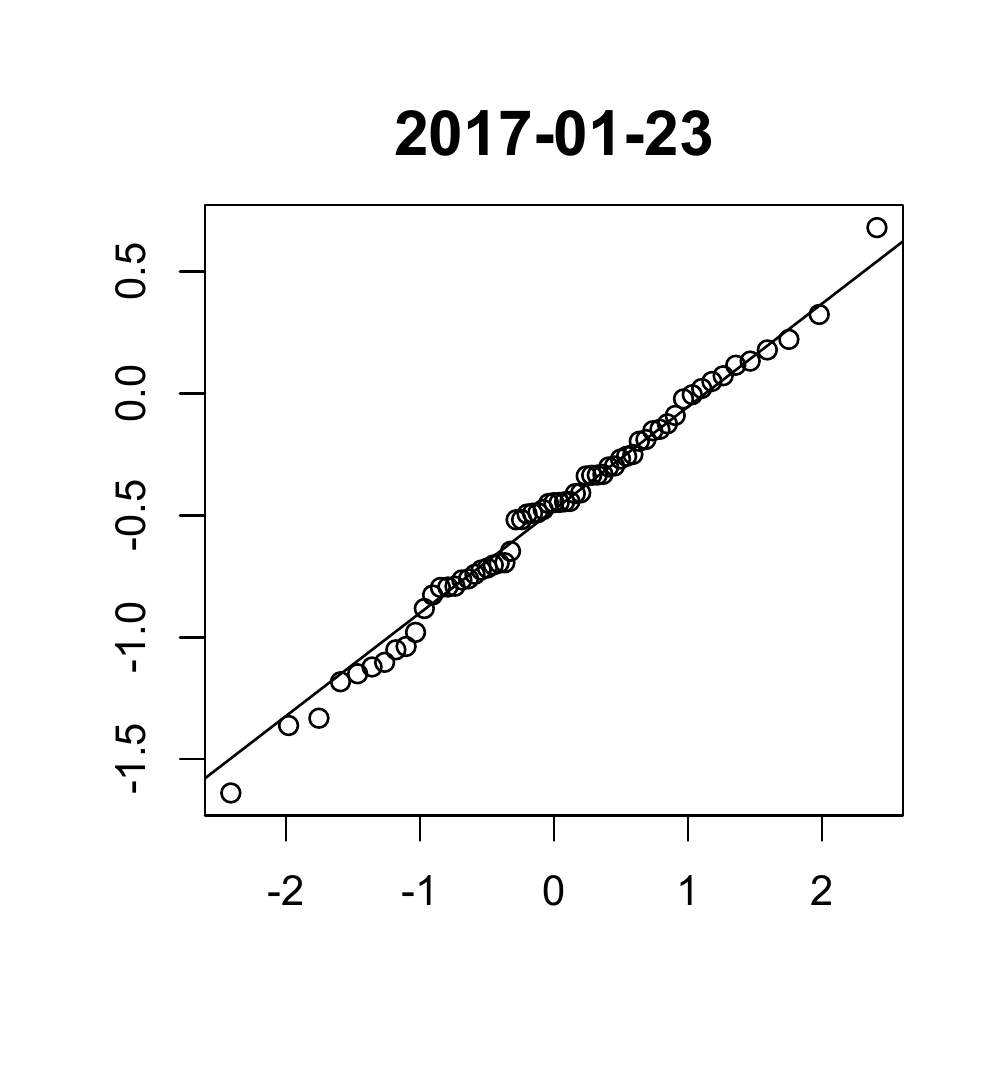} &\hspace{-1.3cm}
 \includegraphics[width=5.7cm, height=6.2cm]{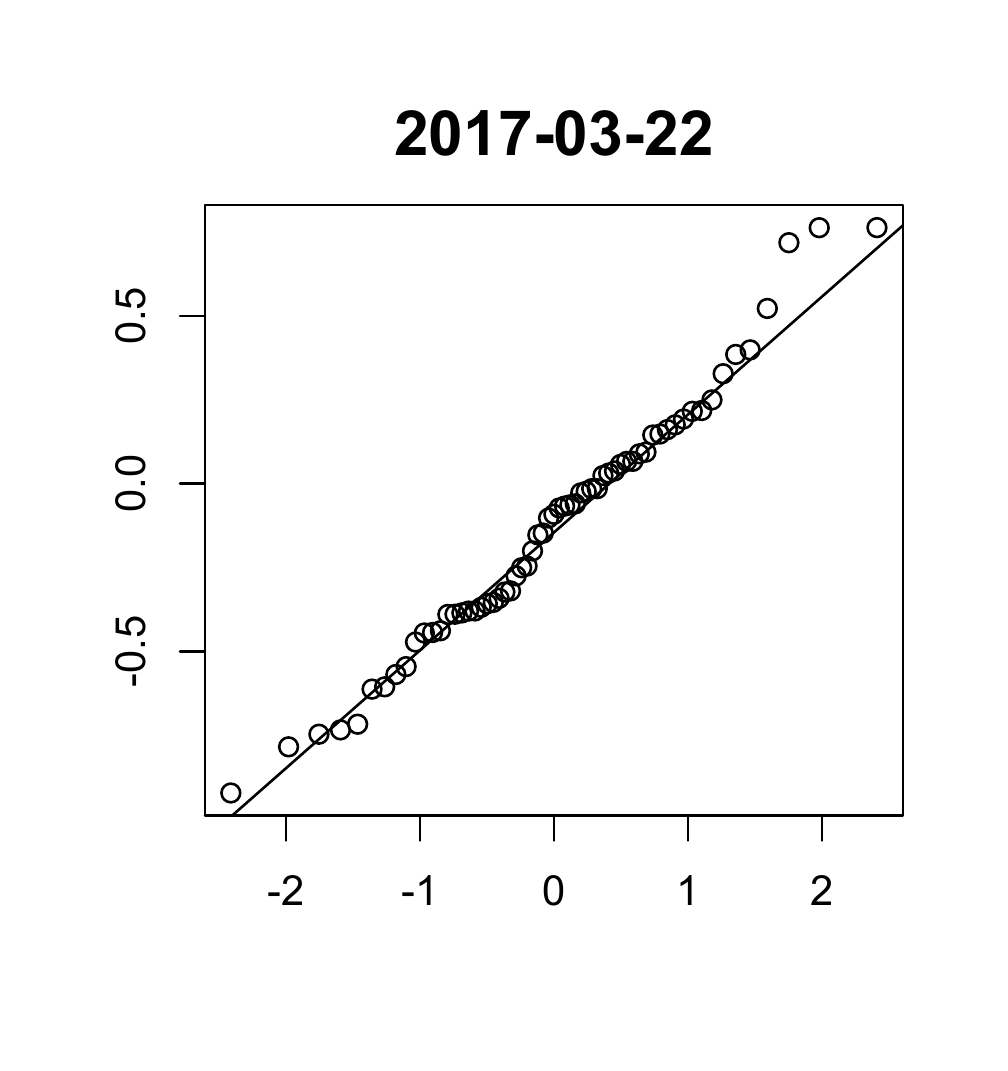} &\hspace{-1.3cm}
 \includegraphics[width=5.7cm, height=6.2cm]{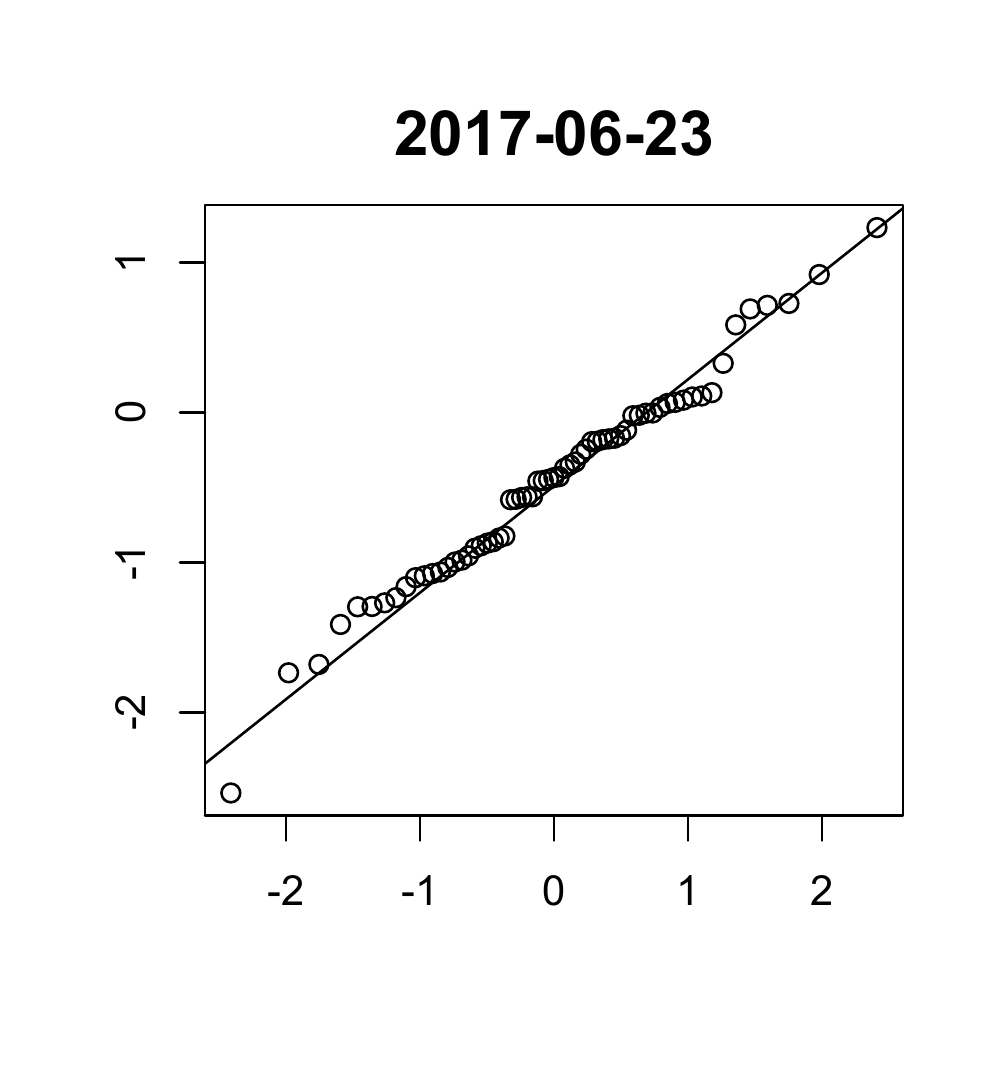}   \\ [-9ex]
\hspace{-.8cm} \includegraphics[width=5.7cm, height=6.2cm]{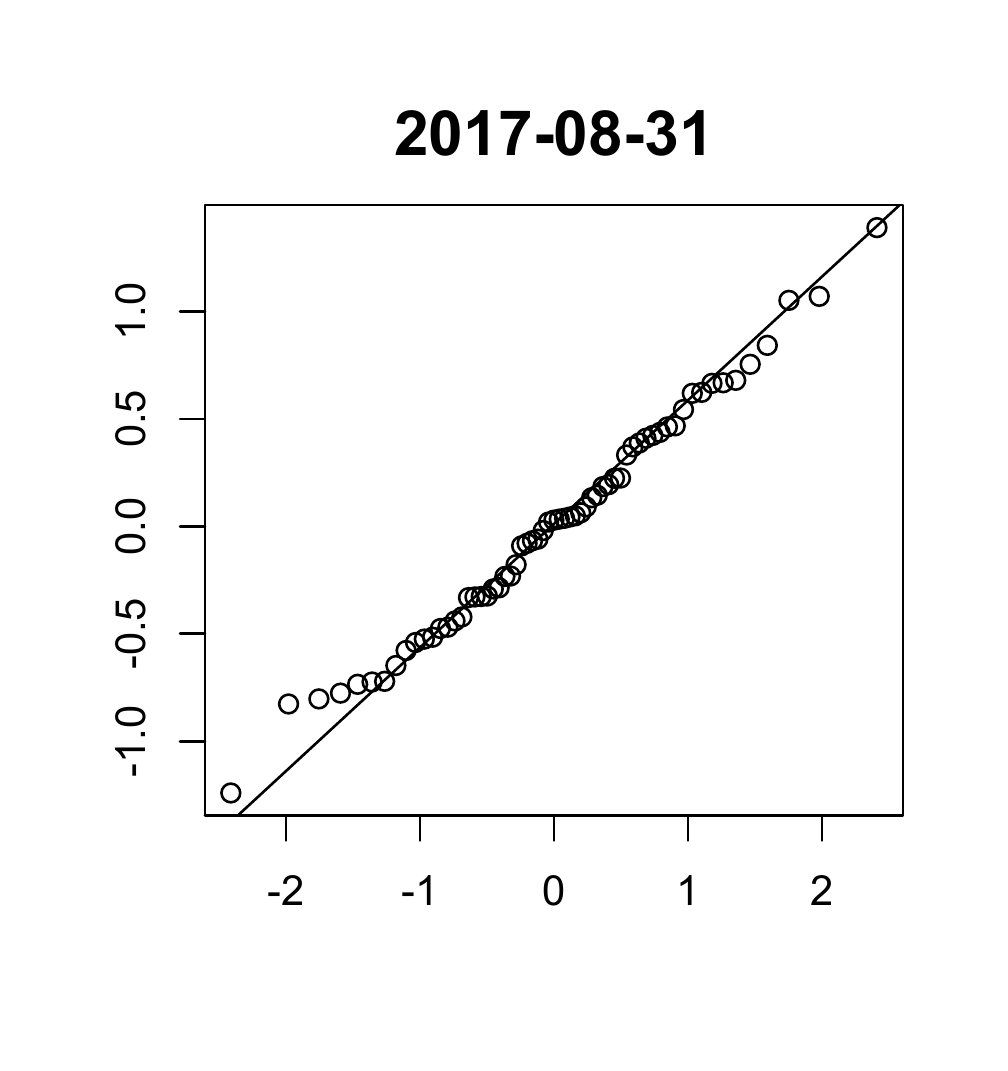} &\hspace{-1.3cm}
 \includegraphics[width=5.7cm, height=6.2cm]{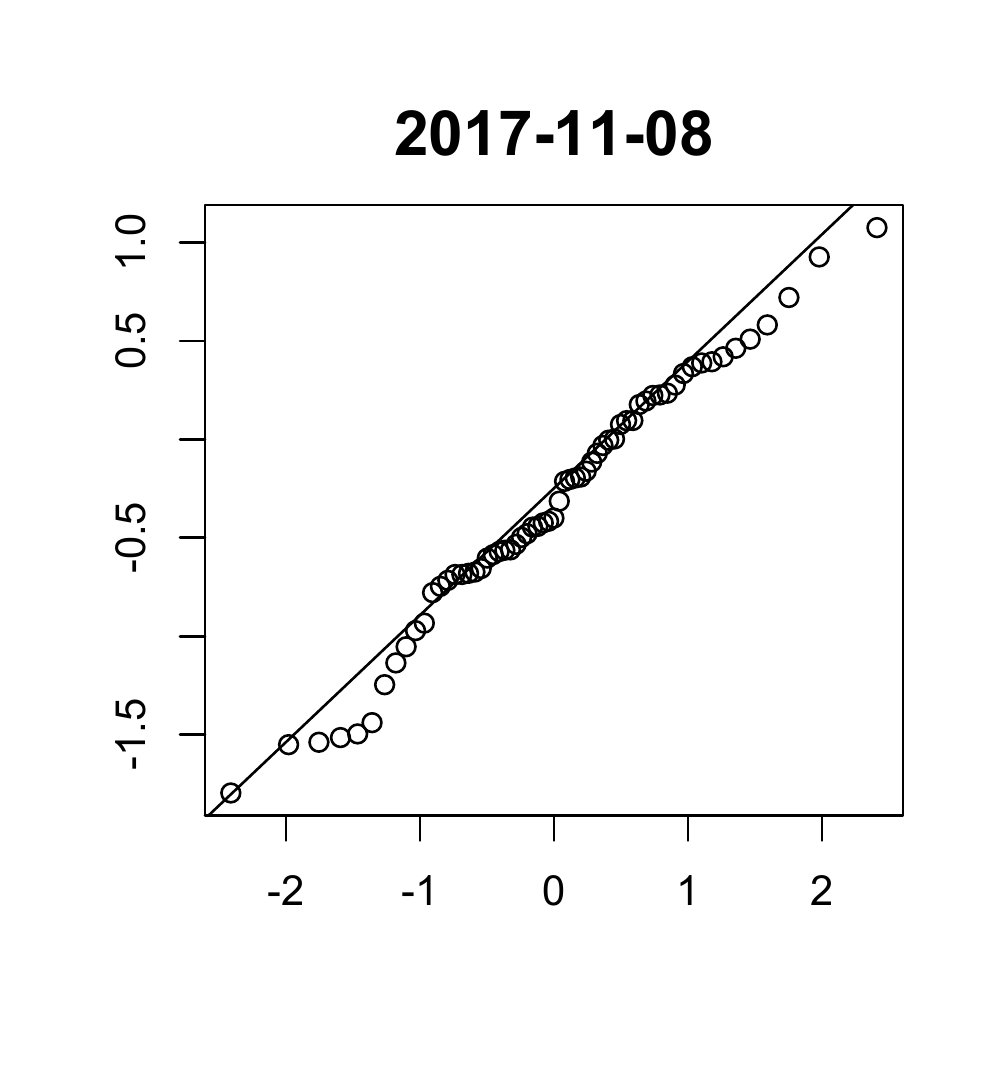} &\hspace{-1.3cm}
 \includegraphics[width=5.7cm, height=6.2cm]{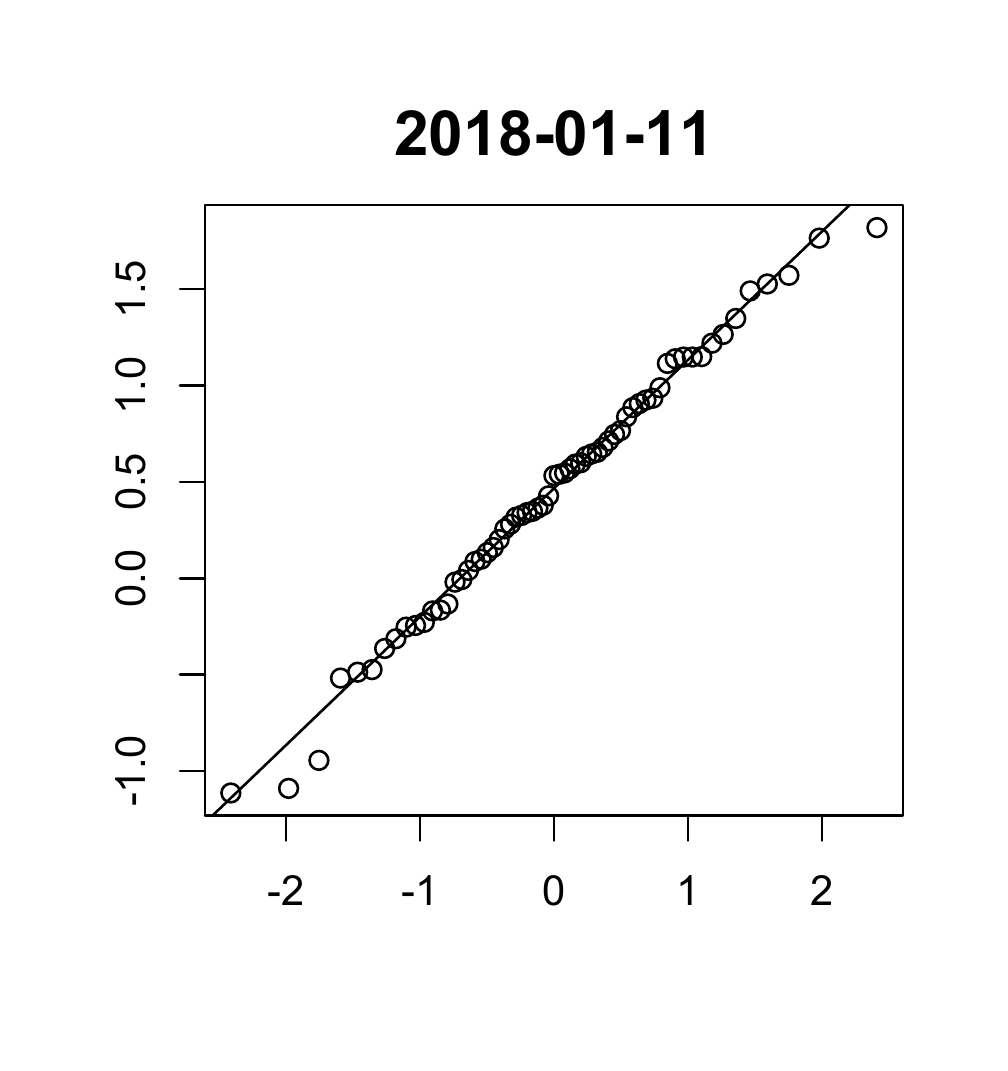}  
\end{tabular}
\vspace{-0.7cm}
\caption{Q-Q plots of logarithmic return for 63 stocks in S\&P 500 Financials Sector.} \label{fig:five}
\end{figure}

\section{Discussion} \label{discussion}

We have revealed an interesting phenomenon regarding the estimability of the correlation parameter in a binary sequence model. Several important directions are left open. One future research is to consider a larger class of models. For instance, suppose the latent sequence $(X_1,\ldots, X_n)$ in \eqref{model:formula} is modeled by the Gaussian copula family \citep{klaassen1997efficient, tsukahara2005semiparametric} with the parameter $\Sigma$ having compound symmetry covariance structure as described in \eqref{compund:symmetry}. What can we say about the estimability of $a^*$? Another direction is to consider that the covariance matrix $\Sigma$ in \eqref{compund:symmetry} is replaced by the more general one in \eqref{network:model} in network modeling. It is clear that both $a^*$ and $b^*$ are nonestimable from the one-bit information of each $X_i$. The question is how much more information is needed to consistently estimate them. Would the phase transition phenomenon we discussed in Remark 3 continue to hold? 

Our results have a few insightful implications as well. For example, modeling a dependent exchangeable binary sequence is subtle. We have shown that estimation is even impossible under a simple one-parameter model. Regarding binary network modeling, it might not be desirable to assume dependency among all the edges. Furthermore, converting a weighted network to a binary one may lose substantial information for the sake of parameter estimation.

\section{Technical details}
\label{proof:part:one}

\subsection{Proof of Theorem 1} \label{proof:thm1}

Denote $A=(A_1,\ldots, A_n)\in\mathbb{R}^n$. The proof follows the general framework of deriving minimax lower bound in \cite{Tsybakov:2008:INE:1522486}. To make it self-contained, we include the standard steps of reducing the problem of lower bounding the estimation error to upper bounding the Kullback--Leibler divergence $D_{KL}({\rm pr}_{a_1}||{\rm pr}_{a_2})$ between ${\rm pr}_{a_1}(A)$ and ${\rm pr}_{a_2}(A)$. For a given estimator $T(A)$, define a classifier $\chi_T(A)=\argmin_{a^*\in \{a_1,a_2\}}|T(A)-a^*|$, and denote a general classifier by $\chi(A)$. Then,
\begin{eqnarray*}
&&\inf_{T}\max_{a^*\in \{a_1,a_2\}}E(R(|T(A)-a^*|)) \\
&\overset{(a)}{\geq}&R(|a_1-a_2|/2) \inf_{T}\max_{a^*\in \{a_1,a_2\}} {\rm pr}_{a^*}(|T(A)-a^*|>|a_1-a_2|/2) \\
&\geq& R(|a_1-a_2|/2) \inf_{T}\max_{a^*\in \{a_1,a_2\}} {\rm pr}_{a^*}(\chi_T(A)\neq a^*)    \\
&\geq& R(|a_1-a_2|/2) \inf_T ({\rm pr}_{a_1}(\chi_T(A)\neq a_1)+{\rm pr}_{a_2}(\chi_T(A)\neq a_2)) \\
&\overset{(b)}{\geq}& 2R(|a_1-a_2|/2) \inf_{\chi}{\rm pr}_{a^*}(\chi(A)\neq a^*) \\
&=& R(|a_1-a_2|/2) \sum_{A\in \{0,1\}^n} \min({\rm pr}_{a_1}(A), {\rm pr}_{a_2}(A)) \\
&\overset{(c)}{\geq}& \frac{1}{2}R(|a_1-a_2|/2) {\rm exp}(-D_{KL}({\rm pr}_{a_1}||{\rm pr}_{a_2})),
\end{eqnarray*}
where step $(a)$ follows by Markov's inequality; the term $\inf_{\chi}{\rm pr}_{a^*}(\chi(A)\neq a^*)$ in $(b)$ is interpreted as Bayes error rate by assuming $a^*$ is uniform on $\{a_1,a_2\}$; and $(c)$ is from Lemma 2.6 in \cite{Tsybakov:2008:INE:1522486}. The rest of the proof is to  upper bound $D_{KL}({\rm pr}_{a_1}||{\rm pr}_{a_2})$. Denote $c_{a_1}=\sqrt{\frac{a_1}{1-a_1}}, c_{a_2}=\sqrt{\frac{a_2}{1-a_2}}$. It holds that $c^{-1}_{a_1}>c^{-1}_{a_2}$ because $0<a_1<a_2<1$. Then Condition \textbf{C} implies that there exists a constant $\ell>0$ such that 
\[
\gamma(c_{a_1}^{-1}z+\tau/\sqrt{a_1})\leq \ell\cdot \gamma(c_{a_2}^{-1}z+\tau/\sqrt{a_2}), \quad \forall z\in \mathbb{R}.
\]
Based on the above result, we then have 
\begin{align*}
&D_{KL}({\rm pr}_{a_1}||{\rm pr}_{a_2})=\sum_{A\in \{0,1\}^n}{\rm pr}_{a_1}(A)\log \frac{{\rm pr}_{a_1}(A)}{{\rm pr}_{a_2}(A)}\\
=&\sum_{A\in \{0,1\}^n}{\rm pr}_{a_1}(A)\log \frac{\int_{-\infty}^{\infty}[1-G(-c_{a_1}z+\tau/\sqrt{1-a_1})]^{\sum_{i}A_i}[G(-c_{a_1}z+\tau/\sqrt{1-a_1})]^{n-\sum_i A_i}\gamma(z)dz}{\int_{-\infty}^{\infty}[1-G(-c_{a_2}z+\tau/\sqrt{1-a_2})]^{\sum_{i}A_i}[G(-c_{a_2}z+\tau/\sqrt{1-a_2})]^{n-\sum_i A_i}\gamma(z)dz} \\
\overset{(d)}{=}& \sum_{A\in \{0,1\}^n}{\rm pr}_{a_1}(A)\log \frac{c_{a_1}^{-1}\int_{-\infty}^{\infty}[1-G(-z)]^{\sum_{i}A_i}[G(-z)]^{n-\sum_i A_i}\gamma(c^{-1}_{a_1}z+\tau/\sqrt{a_1})dz}{c_{a_2}^{-1}\int_{-\infty}^{\infty}[1-G(-z)]^{\sum_{i}A_i}[G(-z)]^{n-\sum_i A_i}\gamma(c^{-1}_{a_2}z+\tau/\sqrt{a_2})dz} \\
\leq&\sum_{A\in \{0,1\}^n}{\rm pr}_{a_1}(A)\log\frac{c_{a_2}\cdot \ell}{c_{a_1}}=\log\frac{c_{a_2}\cdot \ell}{c_{a_1}} <\infty.
\end{align*} 
Here, $(d)$ is  due to a change of variables.

$\hfill \qed$

\subsection{Proof of Theorem \ref{three:thm}} \label{thm2:proof:add}

Denote 
\[
c_{\tau_1,a^*,Y}=G\Big(\frac{\tau_1-\sqrt{a^*}Y}{\sqrt{1-a^*}}\Big), \quad c_{\tau_2,a^*,Y}=1-G\Big(\frac{\tau_2-\sqrt{a^*}Y}{\sqrt{1-a^*}}\Big),
\]
and define a two-dimensional function $H(\cdot,\cdot): [0,1]\times [0,1] \rightarrow \mathbb{R}$,
\[
H(u,v)=1-\bigg[\frac{\tau_1-\tau_2}{G^{-1}(u)-G^{-1}(1-v)}\bigg]^2\cdot \mathbbm{1}_{(u,v)\in \mathcal{B}},
\]
where $\mathcal{B}=\{(u,v)\in [0,1]\times [0,1]: 0<u<1,0<v<1,u+v\neq 1\}$. Then it is straightforward to verify that
\begin{align}
\label{rewrite:form}
\hat{a}_n=H(\bar{A}^{(1)},\bar{A}^{(3)}), \quad a^*=H(c_{\tau_1,a^*,Y}, c_{\tau_2,a^*,Y}).
\end{align}
Moreover, for a given pair $(a,b)\in \mathcal{B}$, define a set 
\begin{align*}
&\Delta_{a,b}=\{(x,y)\in [0,1]\times [0,1]: \sqrt{(x-a)^2+(y-b)^2}\leq \epsilon_{a,b}\}, \\
&\epsilon_{a,b}=\min\Big\{\frac{|a+b-1|}{2\sqrt{2}},\frac{a}{2},\frac{1-a}{2}, \frac{b}{2},\frac{1-b}{2}\Big\}.
\end{align*}
It is direct to confirm that $\Delta_{a,b} \subseteq \mathcal{B}$.
By Hoeffding's inequality, we have $\forall t >0$,
\begin{align*}
 {\rm pr}(|\bar{A}^{(1)}-c_{\tau_1,a^*,Y}|>t)\leq 2\exp(-2nt^2), \quad {\rm pr}( |\bar{A}^{(3)}-c_{\tau_2,a^*,Y}|>t)\leq 2\exp(-2nt^2).
\end{align*}
Hence,
\begin{align}
\label{order:converge}
\bar{A}^{(1)}-c_{\tau_1,a^*,Y}=O_p(n^{-1/2}), \quad \bar{A}^{(3)}-c_{\tau_2,a^*,Y}=O_p(n^{-1/2}).
\end{align}
By \eqref{rewrite:form} we can write 
\begin{align*}
\sqrt{n}(\hat{a}_n-a^*)=&\sqrt{n}\cdot [H(\bar{A}^{(1)},\bar{A}^{(3)})-H(c_{\tau_1,a^*,Y}, c_{\tau_2,a^*,Y})]\cdot \mathbbm{1}_{(\bar{A}^{(1)},\bar{A}^{(3)})\in \Delta_{c_{\tau_1,a^*,Y}, c_{\tau_2,a^*,Y}}}+ \\
&\sqrt{n}\cdot [H(\bar{A}^{(1)},\bar{A}^{(3)})-H(c_{\tau_1,a^*,Y}, c_{\tau_2,a^*,Y})]\cdot \mathbbm{1}_{(\bar{A}^{(1)},\bar{A}^{(3)})\notin \Delta_{c_{\tau_1,a^*,Y}, c_{\tau_2,a^*,Y}}}\\
=&J_1+J_2.
\end{align*}
We now bound the two terms $J_1$ and $J_2$ respectively. For $J_1$, since $(c_{\tau_1,a^*,Y}, c_{\tau_2,a^*,Y})\in \mathcal{B}$, it holds that $\Delta_{c_{\tau_1,a^*,Y}, c_{\tau_2,a^*,Y}}\subseteq \mathcal{B}$. Note that $H(u,v)$ is continuously differentiable over the open set $\mathcal{B}$. Therefore, we can take first-order Taylor expansion of $H(u,v)$ at $(u,v)=(c_{\tau_1,a^*,Y}, c_{\tau_2,a^*,Y})$ to obtain
\begin{align*}
J_1\leq \sqrt{n}\cdot \big|\frac{\partial H(c_1,c_2)}{\partial u}\big |\cdot |\bar{A}^{(1)}-c_{\tau_1,a^*,Y}|+\sqrt{n}\cdot \big|\frac{\partial H(c_1,c_2)}{\partial v}\big |\cdot |\bar{A}^{(3)}-c_{\tau_2,a^*,Y}|,
\end{align*}
where $(c_1,c_2)\in \Delta_{c_{\tau_1,a^*,Y}, c_{\tau_2,a^*,Y}}$. Based on \eqref{order:converge} and the fact that $\big|\frac{\partial H(c_1,c_2)}{\partial u}\big |=O_p(1), \big|\frac{\partial H(c_1,c_2)}{\partial v}\big |=O_p(1)$, we conclude that $J_1=O_p(1)$. Regarding $J_2$, since 
\[
(\bar{A}^{(1)},\bar{A}^{(3)})\overset{a.s.}{\longrightarrow } (c_{\tau_1,a^*,Y}, c_{\tau_2,a^*,Y}),
\]
with probability one we have as $n\rightarrow \infty$,
\[
(\bar{A}^{(1)},\bar{A}^{(3)})\in \Delta_{c_{\tau_1,a^*,Y}, c_{\tau_2,a^*,Y}},
\]
which implies that $J_2=O_p(1)$.
$\hfill \qed$

\section*{Acknowledgement}
This research is partially supported by NSF CAREER grant DMS-1554804. The authors are grateful to Professor Anthony C. Davison and Professor Zhiliang Ying for their insightful comments which greatly improved the scope and presentation of this paper.

\appendix

\section{ Notations and preliminaries}

Recall that under the model \eqref{model:formula2}, $G(\cdot)$ and $p(\cdot)$ are the cumulative distribution function and probability density function of $Y_i$ respectively; $G^{-1}(\cdot)$ is the inverse function of $G(\cdot)$; and $\gamma(\cdot)$ is the probability density function of $Y$. We will use the following notations extensively, 
\begin{eqnarray*}
\bar{A}&=&n^{-1}\sum_{i=1}^nA_i, ~~z^*=-\sqrt{\frac{1-a}{a}}G^{-1}(1-\bar{A}), \\
F_n(z)&=&\bar{A}\log(1-G(-c_az))+(1-\bar{A})\log G(-c_az), ~~c_a=\sqrt{\frac{a}{1-a}}.
\end{eqnarray*}

 The following lemma will be useful in the later proofs.

\begin{lemma}\label{cdf:normal}
The followings hold:
\begin{itemize}
\item[(i)] The likelihood function takes the form
\[
L_n(a;\{A_i\})=\int_{-\infty}^{\infty}\exp(nF_n(z))\cdot \gamma(z+\tau/\sqrt{a})dz.
\]
\item[(ii)] When $\bar{A}\in (0,1)$, the function $F_n(z)$ is quasi-concave, and obtains the maximum at a unique point $z=z^*$.
\item[(iii)] Under Condition \textbf{B}, there exists a constant $\kappa_a$ only depending on $a$ such that 
\[
\sup_{z\in\mathbb{R}}|F^{'''}_n(z)| \leq \kappa_a<\infty.
\]
\item[(iv)] When $\bar{A}\in (0,1)$, we have
\begin{align*}
F_n(z^*)=\bar{A}\log\bar{A}+(1-\bar{A})\log(1-\bar{A}), \quad F''_n(z^*)=\frac{-c^2_a[p(G^{-1}(1-\bar{A}))]^2}{\bar{A}(1-\bar{A})}.
\end{align*}
\end{itemize}

\end{lemma}

\begin{proof}
For Part (i), the likelihood function has been written in \eqref{likelihood:form}. A change of variables leads to the current form. Regarding Part (ii), we first compute the derivative of $F_n(z)$,
\begin{align*}
F'_n(z)=c_a\cdot p(-c_az)\cdot \Big[\frac{\bar{A}}{1-G(-c_az)}-\frac{1-\bar{A}}{G(-c_az)}\Big],
\end{align*}
It is straightforward to confirm that $F'_n(z)$ is positive for $z \in (\infty, z^*)$, and negative for $z\in (z^*,\infty)$. Part (iii) can be directly verified by computing $F^{'''}_n(z)$ under  Condition \textbf{B}. Part (iv) can be checked by some straightforward calculations. 
\end{proof}

\section{Proof of Proposition 1} 

According to Lemma \ref{cdf:normal} Parts (i) and (ii), we can readily have the upper bound, if $0<\bar{A}<1$,
\begin{eqnarray}\label{upper:bound}
\frac{\log L_n(a;\{A_i\})}{n}=\frac{1}{n}\log\int_{-\infty}^{\infty}\exp(nF_n(z))\cdot \gamma(z+\tau/\sqrt{a})dz\leq F_n(z^*).
\end{eqnarray}
Regarding the lower bound, if $0< \bar{A} < 1$, we first have for $\epsilon \in (0,\infty)$,
\begin{eqnarray}\label{lower:bound}
&&\int_{-\infty}^{\infty}\exp(nF_n(z))\cdot \gamma(z+\tau/\sqrt{a})dz  \geq  \int_{z^*-\epsilon}^{z^*+\epsilon}\exp(nF_n(z))\cdot \gamma(z+\tau/\sqrt{a})dz \nonumber \\
&=&\int_{z^*-\epsilon}^{z^*+\epsilon}\exp(nF_n(z^*)+nF''_n(\tilde{z})(z-z^*)^2/2)\cdot \gamma(z+\tau/\sqrt{a})dz,
\end{eqnarray}
where the equality holds by a second-order Taylor expansion and $|\tilde{z}-z^*|\leq  \epsilon$. According to Lemma \ref{cdf:normal} Part (iii), $F''_n(\cdot)$ is Lipschitz continuous with a Lipschitz constant $\kappa_a$. Also, Condition \textbf{A} implies that $\gamma(\cdot)$ is Lipschitz continuous with some constant $L_{\gamma}$. Therefore, choosing $\epsilon=\min\{\frac{ \gamma(z^*+\tau/\sqrt{a})}{2L_{\gamma}},\frac{-F''_n(z^*)}{\kappa_a}\}$, we get for $z\in [z^*-\epsilon, z^*+\epsilon]$,
\begin{align*}
F''_n(\tilde{z})\geq 2 F''_n(z^*), \quad \gamma(z+\tau/\sqrt{a}) \geq \frac{1}{2} \gamma(z^*+\tau/\sqrt{a}),
\end{align*}
which enables us to continue from \eqref{lower:bound} to have
\begin{align}
\label{final:lower:bound:use}
\int_{-\infty}^{\infty}\exp(nF_n(z))\cdot \gamma(z+\tau/\sqrt{a})dz&\geq \frac{1}{2}\gamma(z^*+\tau/\sqrt{a})e^{nF_n(z^*)}\int_{z^*-\epsilon}^{z^*+\epsilon}\exp(nF''_n(z^*)(z-z^*)^2)dz \nonumber \\
&=\frac{\gamma(z^*+\tau/\sqrt{a})e^{nF_n(z^*)}}{2\sqrt{-nF''_n(z^*)}}\int_{-\epsilon\sqrt{-nF''_n(z^*)}}^{\epsilon \sqrt{-nF''_n(z^*)}}e^{-z^2}dz.
\end{align}
The last equality is due to a change of variables. We thus have obtained the lower bound when $0<\bar{A}<1$,
\begin{align}
\label{final:lower:bound}
\frac{\log L_n(a;\{A_i\})}{n} \geq F_n(z^*)-\frac{\log n}{2n} + \frac{1}{n}\log\Big(\frac{\gamma(z^*+\tau/\sqrt{a})}{2\sqrt{-F''_n(z^*)}}\int_{-\epsilon\sqrt{-nF''_n(z^*)}}^{\epsilon \sqrt{-nF''_n(z^*)}}e^{-z^2}dz\Big).
\end{align}
Under  model formulation \eqref{model:formula2}, it is clear that 
\begin{align}
\label{a:bar:limit}
\bar{A}\overset{a.s.}{\longrightarrow} 1- G\Big(\frac{\tau-\sqrt{a^*}Y}{\sqrt{1-a^*}}\Big), \quad \mbox{~~as~} n\rightarrow \infty.
\end{align}
Based on the above result and Lemma \ref{cdf:normal} Part (iv), we have almost surely as $n\rightarrow \infty$,
\begin{align}
\label{convergence:results}
&z^*\rightarrow \frac{\sqrt{1-a}(\sqrt{a^*}Y-\tau)}{\sqrt{a(1-a^*)}}, \quad \mathbbm{1}_{0<\bar{A}<1}\rightarrow 1, \quad F''_n(z^*) \rightarrow \frac{-c_a^2p^2(-c_{a^*}Y+c_{a^*,\tau})}{(1-G(-c_{a^*}Y+c_{a^*,\tau}))\cdot G(-c_{a^*}Y+c_{a^*,\tau})},\nonumber \\
&F_n(z^*)\rightarrow (1-G(-c_{a^*}Y+c_{a^*,\tau}))\cdot \log (1-G(-c_{a^*}Y+c_{a^*,\tau})) \nonumber\\
& \mbox{           }+ G(-c_{a^*}Y+c_{a^*,\tau})\cdot \log G(-c_{a^*}Y+c_{a^*,\tau}), 
\end{align}
where $c_{a^*,\tau}=\frac{\tau}{\sqrt{1-a^*}}$.
These results together with the upper and lower bounds in \eqref{upper:bound} and \eqref{lower:bound} complete the proof.
$\hfill \qed$ \\

\section{Proof of Proposition 2}

We first restrict our analysis to the case $0<\bar{A}<1$. Lemma \ref{cdf:normal} Parts (i) and (iv) show that 
\begin{eqnarray*}
&& L_n(a;\{A_i\})\exp(-n(\bar{A}\log\bar{A}+(1-\bar{A})\log(1-\bar{A}))) \\
&=&\int_{-\infty}^{z^*}\exp(n(F_n(z)-F_n(z^*)))\cdot \gamma(z+\tau/\sqrt{a})dz+\int_{z^*}^{\infty}\exp(n(F_n(z)-F_n(z^*)))\cdot \gamma(z+\tau/\sqrt{a})dz.
\end{eqnarray*}
We focus on the first integral above for now. Since $F_n(z)$ is strictly increasing in $(-\infty, z^*)$ from Lemma \ref{cdf:normal} Part (ii), by a change of variable $F_n(z^*)-F_n(z)=\tilde{z}$ we obtain
\begin{align*}
&\int_{-\infty}^{z^*} \exp(n(F_n(z)-F_n(z^*)))\cdot \gamma(z+\tau/\sqrt{a})dz \\
=&\int_{-\infty}^{z^*-\epsilon} \exp(n(F_n(z)-F_n(z^*)))\cdot \gamma(z+\tau/\sqrt{a}) dz+\int_0^{\Delta(\epsilon)}\exp(-n\tilde{z})\cdot \gamma(z+\tau/\sqrt{a})(F'_n(z))^{-1}d\tilde{z}\\
=&J_1+J_2,
\end{align*}
where $\Delta(\epsilon)=F_n(z^*)-F_n(z^*-\epsilon)$. Lemma \ref{cdf:normal} Part (iii) says that $F''_n(\cdot)$ is Lipschitz continuous with a Lipschitz constant $\kappa_a$. We set $\epsilon=-\frac{F''_n(z^*)}{2\kappa_a}$ and analyze $J_1$ and $J_2$. It is clear that 
\begin{align*}
J_1\leq \exp(n(F_n(z^*-\epsilon)-F_n(z^*)))\cdot \int_{-\infty}^{z^*-\epsilon} \cdot \gamma(z+\tau/\sqrt{a})dz\leq \exp(n(F_n(z^*-\epsilon)-F_n(z^*))).
\end{align*}
According to the results in \eqref{convergence:results}, we can conclude that 
\begin{align}
\label{j1:bound}
J_1=O_p(n^{-1}).
\end{align}
Regarding $J_2$, denote $g(\tilde{z})= \gamma(z+\tau/\sqrt{a})(F'_n(z))^{-1}, h(\tilde{z})=g(\tilde{z})(\tilde{z})^{1/2}$. Then the identity 
\[
g(\tilde{z})=h(0)(\tilde{z})^{-1/2}+(\tilde{z})^{1/2}\int_0^1h'(t\tilde{z})dt
\]
yields that
\begin{align}\label{change:variable}
J_2&=\int_0^{\Delta(\epsilon)}\exp(-n\tilde{z})g(\tilde{z})d\tilde{z} \nonumber \\
&=\frac{h(0)}{\sqrt{n}}\int_0^{n\Delta(\epsilon)}e^{-z}z^{-1/2}dz+\int_0^1\int_{0}^{\Delta(\epsilon)}\exp(-n\tilde{z})h'(t\tilde{z})(\tilde{z})^{1/2}d\tilde{z}dt.
\end{align}
Based on the Taylor expansions 
\begin{eqnarray}\label{taylor:one}
\tilde{z}=F_n(z^*)-F_n(z)=-\frac{1}{2}F''_n(z_1)(z-z^*)^2, \quad F'_n(z)=F_n''(z_2)(z-z^*)
\end{eqnarray}
with $z_1, z_2\in [z,z^*]$, the following holds
\begin{align}
\label{hzero:value}
h(0)&=\lim_{\tilde{z}\rightarrow 0} \gamma(z+\tau/\sqrt{a}) (F'_n(z))^{-1}(\tilde{z})^{1/2} \nonumber \\
&=\lim_{z\nearrow z^*}\cdot \gamma(z+\tau/\sqrt{a}) \frac{[-\frac{1}{2}F''_n(z_1)(z-z^*)^2]^{1/2}}{F''_n(z_2)(z-z^*)}=\frac{ \gamma(z^*+\tau/\sqrt{a})}{(-2F''_n(z^*))^{1/2}}.
\end{align}
Next we bound the second term on the right-hand side of \eqref{change:variable}. With some simple calculations we obtain
\begin{eqnarray}\label{second:bound}
h'(\tilde{z})(\tilde{z})^{1/2}&=&-\gamma'(z+\tau/\sqrt{a})\cdot \frac{F_n(z^*)-F_n(z)}{(F'_n(z))^2}+ \nonumber \\
&&\gamma(z+\tau/\sqrt{a})\cdot \frac{F''_n(z)(F_n(z^*)-F_n(z))+\frac{1}{2}(F'_n(z))^2}{(F'_n(z))^3}.
\end{eqnarray}
Since $z\in [z^*-\epsilon, z^*]$ when $\tilde{z}\in [0,\Delta(\epsilon)]$, we have
\begin{align*}
|F''_n(z)-F''_n(z^*)|\leq \kappa_a |z-z^*|\leq \kappa_a \epsilon=\frac{-1}{2}F''_n(z^*),
\end{align*}
thus for any $z \in [z^*-\epsilon, z^*]$,
\begin{align}
\label{lower:bound:1} 
\frac{3}{2}F''_n(z^*)\leq F''_n(z) \leq \frac{1}{2}F''_n(z^*)
\end{align}
Then using the expansions we had about $F_n(z^*)-F_n(z)$ and $F'_n(z)$ in \eqref{taylor:one}, we can obtain
\begin{eqnarray}\label{bound:first}
\sup_{\tilde{z}\in [0,\Delta(\epsilon)]} \frac{|\gamma'(z+\tau/\sqrt{a})(F_n(z^*)-F_n(z))|}{(F'_n(z))^2}=\sup_{\tilde{z}\in [0,\Delta(\epsilon)]} \frac{ |\gamma'(z+\tau/\sqrt{a})F''_n(z_1)|}{2(F''_n(z_2))^2} \leq \frac{3L_{\gamma}}{-F''_n(z^*)},
\end{eqnarray}
where $L_{\gamma}$ denotes the upper bound for $\sup_z |\gamma'(z)|$ from Condition \textbf{A}.
 To bound the other term in \eqref{second:bound}, we need to take the Taylor expansions to a higher order,
\begin{align*}
&F_n(z^*)-F_n(z)=-\frac{1}{2}F''_n(z^*)(z-z^*)^2-\frac{1}{6}F'''_n(z_3)(z-z^*)^3, ~~z_3 \in [z,z^*],\\
&F'_n(z)=F''_n(z^*)(z-z^*)+\frac{1}{2}F'''_n(z_4)(z-z^*)^2, ~~z_4\in [z,z^*], \\ 
& F''_n(z)=F''_n(z^*)+F'''_n(z_5)(z-z^*), ~~z_5 \in [z, z^*].
\end{align*}
Plugging the above expansions into the numerator of the second term in \eqref{second:bound} and use $F'_n(z)=F_n''(z_2)(z-z^*)$ in the denominator, it is not hard to obtain
\begin{align*}
& \sup_{\tilde{z}\in [0,\Delta(\epsilon)]} \Big|\gamma(z+\tau/\sqrt{a})\cdot \frac{F''_n(z)(F_n(z^*)-F_n(z))+\frac{1}{2}(F'_n(z))^2}{(F'_n(z))^3}\Big| \\
\overset{(a)}{\leq}& \sup_{\tilde{z}\in [0,\Delta(\epsilon)]} \frac{\gamma(z+\tau/\sqrt{a})\cdot(C_1\kappa_a^2|z-z^*|+C_2\kappa_a|F''_n(z^*)|)}{|F''_n(z_2)|^3} \overset{(b)}{\leq} \frac{C_3(-L_{\gamma}F''_n(z^*)+2\kappa_a \gamma(z^*+\tau/\sqrt{a}))}{[F''_n(z^*)]^2},
\end{align*}
where $C_i~(i=1,2,3)$ are absolute positive constants. To derive $(a)$ we have used Lemma \ref{cdf:normal} Part (iii); $(b)$ is due to \eqref{lower:bound:1} and the Lipschitz continuity of $\gamma(z)$ implied by Condition \textbf{A}. 

Combining the above upper bound with \eqref{second:bound} and \eqref{bound:first} gives
\begin{eqnarray*}\label{maximum:upper}
\sup_{\tilde{z}\in [0,\Delta(\epsilon)]} |h'(\tilde{z})(\tilde{z})^{1/2}| \leq  \frac{3L_{\gamma}}{-F''_n(z^*)}+\frac{C_3(-L_{\gamma}F''_n(z^*)+2\kappa_a \gamma(z^*+\tau/\sqrt{a}))}{[F''_n(z^*)]^2}.
\end{eqnarray*}
Therefore,
\begin{eqnarray}
&&\Big|\int_0^1\int_{0}^{\Delta(\epsilon)}\exp(-n\tilde{z})h'(t\tilde{z})(\tilde{z})^{1/2}d\tilde{z}dt\Big| \nonumber \\
&\leq & \sup_{\tilde{z}\in [0,\Delta(\epsilon)]} |h'(\tilde{z})(\tilde{z})^{1/2}| \cdot \int_0^1t^{-1/2}dt \cdot \int_0^{\infty}\exp(-n\tilde{z})d\tilde{z}=O_p(n^{-1}). \label{big:o:term}
\end{eqnarray}
Putting together the results \eqref{j1:bound}, \eqref{change:variable}, \eqref{hzero:value} and \eqref{big:o:term}, we have shown that if $0<\bar{A}<1$,
\begin{eqnarray*}
&&\int_{-\infty}^{z^*}\exp(n(F_n(z)-F_n(z^*)))\cdot \gamma(z+\tau/\sqrt{a})dz \\
&=& \frac{\gamma(z^*+\tau/\sqrt{a})}{\sqrt{-2nF''_n(z^*)}} \int_0^{n\Delta(\epsilon)}e^{-z}z^{-1/2}dz  +O_p(n^{-1}),
\end{eqnarray*}
where $\epsilon=\frac{-F''_n(z^*)}{2\kappa_a}, \Delta(\epsilon)=F_n(z^*)-F_n(z^*-\epsilon)$. Using similar arguments we can obtain an analogous result for the other integral
\begin{align*}
&\int_{z^*}^{\infty}\exp(n(F_n(z)-F_n(z^*)))\cdot \gamma(z+\tau/\sqrt{a})dz \\
=&\frac{\gamma(z^*+\tau/\sqrt{a})}{\sqrt{-2nF''_n(z^*)}} \int_0^{n\Delta(-\epsilon)} e^{-z}z^{-1/2}dz  +O_p(n^{-1}).
\end{align*}
Moreover, given that $ \bar{A}\overset{a.s.}{\rightarrow} 1- G(\frac{\tau-\sqrt{a^*}Y}{\sqrt{1-a^*}}) \in (0,1)$ from \eqref{a:bar:limit}, we can conclude that
\begin{align*}
&L_n(a;\{A_i\}) \exp(-n(\bar{A}\log\bar{A}+(1-\bar{A})\log(1-\bar{A}))) \\
=&\underbrace{\frac{\gamma(z^*+\tau/\sqrt{a})}{\sqrt{-2nF''_n(z^*)}} \cdot \Big[\int_0^{n\Delta(\epsilon)}e^{-z}z^{-1/2}dz+ \int_0^{n\Delta(-\epsilon)} e^{-z}z^{-1/2}dz\Big] \cdot \mathbbm{1}_{0<\bar{A}<1}}_{:=I}+O_p(n^{-1}).
\end{align*}
The remainder of the proof is to show
\begin{align}
\label{final:piece}
I= \frac{\gamma(c_{a^*}c^{-1}_aY+c_{a,a^*,\tau})}{c_a}\cdot \sqrt{\frac{2\pi\cdot (1-G(c_{a^*,\tau,Y}))\cdot G(c_{a^*,\tau,Y})}{np^2(c_{a^*,\tau,Y})}}+O_p(n^{-1}),
\end{align}
where $c_{a,a^*,\tau}=\frac{(\sqrt{1-a^*}-\sqrt{1-a})\tau}{\sqrt{a(1-a^*)}}, c_{a^*,\tau,Y}=\frac{\tau-\sqrt{a^*}Y}{\sqrt{1-a^*}}$.
Towards that goal, define the function $H: (0,1)\rightarrow \mathbb{R}$
\[
H(x)=\frac{\sqrt{x(1-x)}\cdot \gamma(-G^{-1}(1-x)c^{-1}_a+\tau/\sqrt{a})}{p(G^{-1}(1-x))}.
\]
Then it is direct to confirm that \eqref{final:piece} is equivalent to 
\begin{align*}
O_p(n^{-1})=&\frac{\mathbbm{1}_{0<\bar{A}<1}}{c_a\sqrt{2n}}\big[H(\bar{A})-H(1-G(c_{a^*,\tau,Y}))\big]\cdot \Big[\int_0^{n\Delta(\epsilon)}e^{-z}z^{-1/2}dz+ \int_0^{n\Delta(-\epsilon)} e^{-z}z^{-1/2}dz\Big] -\\
&\frac{\mathbbm{1}_{0<\bar{A}<1}}{c_a\sqrt{2n}}H(1-G(c_{a^*,\tau,Y}))\cdot \Big[\int_{n\Delta(\epsilon)}^{\infty}e^{-z}z^{-1/2}dz+ \int_{n\Delta(-\epsilon)}^{\infty} e^{-z}z^{-1/2}dz\Big]=I_1-I_2.
\end{align*} 
We now show that the two terms $I_1$ and $I_2$ are both of order $O_p(n^{-1})$. We first have
\[
|I_1|\leq \frac{2\sqrt{\pi}\mathbbm{1}_{0<\bar{A}<1}}{c_a\sqrt{2n}}\cdot |H(\bar{A})-H(1-G(c_{a^*,\tau,Y}))|.
\]
By Hoeffding's inequality, 
\[
{\rm pr}(\sqrt{n}|\bar{A}-1+G(c_{a^*,\tau,Y})|>t)\leq 2\exp(-2t^2), \quad \forall t>0.
\]
So $\bar{A}-1+G(c_{a^*,\tau,Y})=O_p(n^{-1/2})$. Since $H(\cdot)$ is continuously differentiable, taking first-order Taylor expansion for $H(\cdot)$ leads to $I_1=O_p(n^{-1})$. Regarding $I_2$, since
\begin{align*}
\int_{n\Delta(\epsilon)}^{\infty}e^{-z}z^{-1/2}dz+ \int_{n\Delta(-\epsilon)}^{\infty} e^{-z}z^{-1/2}dz \leq \sqrt{2\pi}\big[\exp(-n\Delta(\epsilon)/2)+ \exp(-n\Delta(-\epsilon)/2)\big],
\end{align*}
and $\Delta(\pm \epsilon)$ converge to some non-degenerate random variables, it is clear that $I_2=O_p(n^{-1})$.
$\hfill \qed$

\bibliography{mc}
\bibliographystyle{ims}

\end{document}